%% file: HS.tex
\documentclass[11pt,reqno]{amsart}
\include{Preamble}
\makeatother

\date{}
\begin{document}
\title{Admissible function spaces for weighted Sobolev inequalities}


\author{T. V. Anoop$^{1*}$}
\address{$^{1*}$corresponding author and also supported by the INSPIRE Research grant DST/INSPIRE/04/2014/001865.}

\author{Nirjan Biswas$^1$}
\address{$^1$ Department of Mathematics, Indian Institute of Technology Madras, Chennai, 600036, India.}
\email{anoop@\text{iitm}.ac.in, nirjaniitm@gmail.com}

\author{Ujjal Das$^2$}
\address{$^2$The Institute of Mathematical Sciences, HBNI, Chennai, 600113, India.}
\email{ujjaldas@imsc.res.in}
\maketitle
\begin{abstract}
Let $k,N\in \mathbb{N}$ with  $1\le k\le N$ and let $\Omega=\Omega_1 \times \Omega_2$ be an open set in $\mathbb{R}^k \times \mathbb{R}^{N-k}$. For $p\in (1,\infty)$ and  $q \in (0,\infty),$ we consider the following weighted Sobolev type inequality:
\begin{align}\label{pq}
   \int_{\Omega} |g_1(y)||g_2(z)| |u(y,z)|^q \, {\rm d}y {\rm d}z   \leq C \left( \int_{\Omega} | \nabla u(y,z) |^p \, {\rm d}y {\rm d}z \right)^{\frac{q}{p}},   \quad \forall \, u \in \mathcal{C}^1_c(\Omega),
\end{align}
for some $C>0$. Depending on the values of $N,k,p,q$ we have identified various pairs  of Lorentz spaces, Lorentz-Zygmund spaces and weighted Lebesgue spaces for $(g_1, g_2)$ so that \eqref{pq} holds. Furthermore, we give a sufficient condition on $g_1,g_2$ so that the best constant in \eqref{pq} is attained in the Beppo-Levi space $\mathcal{D}^{1,p}_0(\Omega)$-the completion of $\mathcal{C}^1_c(\Omega)$ with respect to $\norm{\nabla u}_{L^p(\Omega)}$. 
\end{abstract}


\tableofcontents
\input{Introduction}
\input{Preliminaries}
\input{Hpq}
\input{Inequality}

\input{Compact}
\input{ExampleandRemarks}
\appendix
\input{Appendix}

\bibliographystyle{plainurl}

\end{document}

%% file: Preamble.tex
\usepackage{footnote}
\usepackage{graphicx}
\usepackage{xcolor}
\usepackage{cmap,mathtools}
\usepackage[T1]{fontenc}
\usepackage[utf8]{inputenc}
\usepackage[english]{babel}
\usepackage{yhmath,amsmath,amsfonts,amssymb,enumerate,amsthm,appendix,caption,lipsum,bm}
\usepackage{url}
\usepackage[shortlabels]{enumitem}

\allowdisplaybreaks

\usepackage{graphicx}
\usepackage{epstopdf}
\usepackage{a4wide}
\usepackage{xparse}
\usepackage{xcolor}
\usepackage[colorlinks=true,linktocpage,pdfpagelabels,bookmarksnumbered,bookmarksopen]{hyperref}
\definecolor{ForestGreen}{rgb}{0.1,0.6,0.05}
\definecolor{EgyptBlue}{rgb}{0.063,0.1,0.6}
\hypersetup{
	colorlinks=true,
	linkcolor=EgyptBlue,         
	citecolor=ForestGreen,
	urlcolor=olive
}

\usepackage[hyperpageref]{backref}

\usepackage[margin=0.95in]{geometry}

\newtheorem{theorem}{Theorem}
\newtheorem{proposition}[theorem]{Proposition}
\newtheorem{lemma}[theorem]{Lemma}
\newtheorem{corollary}[theorem]{Corollary}
\theoremstyle{definition}
\newtheorem{definition}[theorem]{Definition}
\newtheorem{remark}[theorem]{Remark}
\newtheorem{example}[theorem]{Example}

\let\OLDthebibliography\thebibliography
\renewcommand\thebibliography[1]{
	\OLDthebibliography{#1}
	\setlength{\parskip}{1pt}
	\setlength{\itemsep}{1pt plus 0.3ex}
}

\numberwithin{equation}{section}
\numberwithin{theorem}{section}
\numberwithin{equation}{section}
\numberwithin{theorem}{section}

\DeclarePairedDelimiter\abs{\lvert}{\rvert}%
\DeclarePairedDelimiter\norm{\lVert}{\rVert}%
\makeatletter
\let\oldnorm\norm
\def\norm{\@ifstar{\oldnorm}{\oldnorm*}}

\DeclareMathOperator*{\esssup}{ess\,sup}
\DeclareMathOperator*{\lowlim}{\underline{lim}}

\newcommand{\al} {\alpha}

\newcommand{\pa} {\partial}
\newcommand{\be} {\beta}
\newcommand{\de} {\delta}
\newcommand{\De} {\Delta}

\newcommand{\ga} {\gamma}
\newcommand{\Ga} {\Gamma}
\newcommand{\om} {\omega}
\newcommand{\Om} {\Omega}

\newcommand{\la} {\lambda}

\newcommand{\Gr} {\nabla}

\newcommand{\no} {\nonumber}
\newcommand{\noi} {\noindent}
\newcommand{\vph} {\varphi}

\newcommand{\ep} {\epsilon}
\newcommand{\ra} {\rightarrow}

\newcommand{\wra} {\rightharpoonup}

\newcommand{\pq}{\mathcal{H}_{p,q}(\Om)}

\newcommand{\pp}{\mathcal{H}_{p,p}(\Om)}

\newcommand{\pqR}{\mathcal{H}_{p,q}(\mathbb{R}^N)}

\newcommand{\dpp}{\D_0^{1,p}}
\newcommand{\Dp}{\D_0^{1,p}(\Om)}
\newcommand{\DN}{\D_0^{1,N}(\Om)}

\newcommand{\intRn}{\displaystyle{\int_{\mathbb{R}^N}}}

\newcommand{\intRnk}{\displaystyle{\int_{\mathbb{R}^{N-k}}}}

\newcommand\restr[2]{{
  \left.\kern-\nulldelimiterspace 
  #1 
  \right|_{#2} 
  }}

\usepackage[hyperpageref]{backref}

\NewDocumentCommand{\intervals}{ m m }{#1 {,} #2}

\def\wp{{W^{1,p}_0(\Om)}}

\def\dpR{{\D^{1,p}_0(\RN)}}

\def\C{{\mathcal C}}
\def\D{{\mathcal D}}

\def\N{{\mathbb N}}
\def\F{{\mathcal F}}
\def\M{{\mathcal M}}
\def\S{\mathbb{S}}

\def\cset{{\subset \subset }}
\def\R{{\mathbb R}}
\def\RN{{\mathbb R}^N}
\def\Rk{{\mathbb R}^k}
\def\RNk{{\mathbb R}^{N-k}}

\def\({{\Big(}}
\def\){{\Big)}}

\def\ws2{{\F_{\frac{N}{2}}}}
\def\cc{{\C_c^\infty}}
\def\c1{{\C_c^1}}
\def\dis{{\displaystyle \int_{\Omega}}}

\def\disone{{\displaystyle \int_{\Omega_1}}}

\def\p{{p^{\prime}}}
\def\q{{q^{\prime}}}

\def\f{{\tilde{f}}}
\def\g{{\tilde{g}}}
\def\h{{\tilde{h}}}
\def\d{{\rm d}}
\def\dr{{\rm d}r}
\def\dS{{\rm d}S}

\def\dl{{\rm d}\lambda}
\def\ds{{\rm d}s}
\def\dt{{\rm d}t}
\def\dz{{\rm d}z}
\def\dx{{\rm d}x}
\def\dy{{\rm d}y}

\def\vpp{{\vph^{\prime}}}

\def\H{{\mathcal{H}}}
\def\h{{\tilde{h}}}

\usepackage[foot]{amsaddr}
\usepackage[pagewise]{lineno}

%% file: Introduction.tex
\section{Introduction and the Main Results}
Let $k,N\in \N$ be such that  $1\le k\le N.$  For an open set $\Om$ in $\R^N$ and $g\in L^1_{loc}(\Om)$, we assume  the following:
\begin{equation}\label{products}
 \left.
\begin{aligned}
    \bullet\;  &\Om = \Om_1 \times \Om_2, \text{ where } \Om_1 \text{ and } \Om_2 \text{ are open sets in } \R^k \text{ and } \R^{N-k} \text{ respectively},\\
    \bullet\;   &g(x)=g_1(y) g_2(z),\; x:=(y,z) \in \Om_1 \times \Om_2, \text{ where } g_1\in L^1_{loc}(\Om_1) \text{ and } g_2\in L^1_{loc}(\Om_2),\\
   \bullet\;    &\textit{If $k=N$, then  $\Om=\Om_1$ and $g=g_1.$}
 \end{aligned}
 \right\}\tag{\bf A}
\end{equation}
  Let $p\in (1,\infty)$ and  $q \in (0,\infty).$ For $\Om$ and $g$ as given in \eqref{products}, we look for sufficient conditions on $g_1,g_2$ so that  the following weighted Sobolev type inequality holds:
\begin{align}\label{p-qHardy}
 \int_{\Om} |g(x)| |u(x)|^q \, \dx   \leq C \left( \int_{\Om} | \Gr u(x) |^p \, \dx \right)^{\frac{q}{p}},   \quad \forall \, u \in \c1(\Om),
\end{align}
for some $C>0.$

\begin{definition} [\bm{$(p,q)$} {\bf-Hardy potential}]

\begin{enumerate}
\item[]\mbox{}
\item  A weight function $g\in L^1_{loc}(\Om)$ satisfying \eqref{p-qHardy} is said to be a {$(p,q)$}{-Hardy potential}. The set of all $(p,q)$- Hardy potentials is denoted by $\pq.$ i.e.,
$$\pq=\{g \in L^1_{loc}(\Om): g \textit{ is a $(p,q)$-Hardy potential }\}.$$ 
    \item If  $g$ is of the form $g(x)=g_1(y)g_2(z)$ for some $g_1$ and $g_2$, then we say $g$ is  a {\bf cylindrical potential}. If $g$ is not a cylindrical potential, then we say $g$ is a {\bf non-cylindrical potential}. 
    \end{enumerate}
\end{definition}

It is not difficult to produce examples of weight functions $g$ in $\pq.$ For example, if $\Om_1$ or $\Om_2$ is bounded in one direction, then the Poincar\'{e} inequality shows that $$L^\infty(\Om)\subset \pp.$$ For $N>p$, let $p^*:=\frac{Np}{N-p}$. Then for $\Om\subset \R^N$ with $N>p,$ the Sobolev inequality 
\begin{align}\label{Sobolev}
 \int_{\Om}  |u(x)|^{p^*} \, \dx   \leq C \left( \int_{\Om} | \Gr u(x) |^p \, \dx \right)^{\frac{p^*}{p}},   \quad \forall  \, u \in \c1(\Om)
\end{align}
ensures  $L^\infty(\Om)\subset \H_{p,p^*}(\Om)$. Furthermore, using the  duality of the Lebesgue spaces together with the H\"{o}lder's inequality will give $L^{\frac{N}{p}}(\Om)\subset \pp.$ 
In the literature, there are many existing results that provide various sufficient conditions for $g$ to be a $(p,q)$-Hardy potential. Before we discuss some of them, we introduce two functions that will be appearing  more frequently in this manuscript. 

\noi \underline{$\bm{\al(p,q)}$:}
 For any $N,p,q,$ we define $$ \al(p,q) := \frac{Np}{N(p-q)+ qp}.$$
    Notice that, $\al(p,p)=\frac{N}{p}$ and for $N>p$,  $$\al(p,q) =\frac{p^*}{p^* - q}=\left(\frac{p^*}{q}\right)', \; \forall\, q\in (0,p^*].$$

    \noi \underline{$\bm{P^*(s)}$:} For $N>p$ and $s\in [0,p],$ we define $$P^*(s) :=\frac{p(N-s)}{N-p}.$$
    Observe that, $P^*(0)=p^*, P^*(p)=p$, and for $N>p$, 
    $$ \al(p,P^*(s))= \frac{N}{s},  \quad P^*\left(\frac{N}{\al(p,q)}\right)=q.$$

\subsection{Various sufficient conditions for $(p,q)$-Hardy Potentials} Now we discuss various sufficient conditions for a $(p,q)$-Hardy Potential available in the literature.  

\noi{\bf (i) Fefferman-Phong type conditions:} In  \cite{Fefferman-Phong}, for  $V\in L^1_{loc}(\R^N)$ with $V\le 0,$ Fefferman-Phong estimated the lowest bound of the Schrodinger operator {$-\Delta + V $}. Their result  ensures that there exists $C>0$ such that $$\int_{\R^N}|\Gr u(x)|^2\dx + \int_{\R^N}V(x)|u(x)|^2 \dx\ge -C E_{big} \int_{\R^N}|u(x)|^2\dx,\; \forall u \in \c1(\R^N),$$  where  
$$E_{big}:=\sup_{Q}\left[\left(\frac{1}{|Q|}\int_{Q}|V(x)|^s\dx\right)^{\frac{1}{s}}-c\;\text{diam}(Q)^{-2}\right] \quad \text{ with } c>0 \text{ and } s>1,$$ 
where $Q$ ranges over cubes in $\R^N$ with sides parallel to the axes, $|Q|$ and  diam$(Q)$ are respectively the measure and the diameter of $Q$. Later, in the definition of $E_{big}$, Fefferman  \cite[Theorem 5]{Fefferman} replaced the cubes with the balls as bellow:
$$
E_{big}:=\sup_{\{x\in \R^N,r>0\}}\left[\left(\frac{1}{|B_r(x)|}\int_{B_r(x)}|V(x)|^s \, \dx\right)^{\frac{1}{s}} - c r^{-2}\right] \text{  with } c>0 \text{ and } s>1.$$
Now if  there exists $ c>0,$ and $s>1$  such that
\begin{equation}\label{FF1}\left(\frac{1}{|Q|}\int_{Q}|V(x)|^s\dx\right)^{\frac{1}{s}}\le c\;\text{diam}(Q)^{-2}, \quad \forall\; Q,
\end{equation}
or 
\begin{equation}\label{FF2}\left(\frac{1}{|B_r(x)|}\int_{B_r(x)}|V(x)|^s\dx\right)^{\frac{1}{s}}\le c\; r^{-2},\quad \forall x\in \R^N, \forall \, r>0,
\end{equation}
then  $E_{big}\le 0$  and hence 
$$\intRn |V(x)|u(x)^2 \, \dx \le \intRn |\Gr u(x)|^2 \, \dx, \quad \forall \, u \in \cc(\RN).$$
In particular, if for some $c>0$ and $s>1$, $V$ satisfies \eqref{FF1} or \eqref{FF2}, then by applying the result of Fefferman-Phong to $-\De-|V|$ we get $V\in \H_{2,2}(\R^N).$

Next, we see  Fefferman-Phong type conditions for general $p$ and $q$  via weighted inequalities for the  fractional integrals.
For $N\ge 3$ and $u\in \C^\infty_c(\R^N),$ the Newtonian potential $\Ga*(\De u)$ of $\De u$   coincide with $u$ (\cite[Theorem 2, Pg-147]{Triebel}), where $\Ga$ is the Fundamental solution of the Laplacian. Thus using the integration by parts we get
\begin{align*}
    u(x)=\frac{1}{N\om_N(2-N)}\int_{\R^N}\frac{\De u(y)}{|x-y|^{N-2}}\dy= \frac{1}{N\om_N} \int_{\R^N}\frac{(x-y)\cdot\Gr  u(y)}{|x-y|^{N}}\dy.
    \end{align*}
    Therefore, 
    \begin{equation}\label{esti_conv}
     |u(x)|\le \frac{1}{N\om_N} \int_{\R^N}\frac{|\Gr  u(y)|}{|x-y|^{N-1}}\dy= \frac{1}{N\om_N} I_1(|\Gr u|),   
    \end{equation}    where $I_{\gamma}$ is the Riesz potential operator defined as $$ I_{\gamma}(f)(x) := \int_{\RN} \frac{|f(y)|}{|x-y|^{N-{\gamma}}} \, \dy; \ \gamma\in (0,N).$$
 From \eqref{esti_conv} it is evident that \eqref{p-qHardy} holds, if the following weighted inequality holds:
 \begin{align}\label{Weighted-Riesz}
\int_{\R^N} \abs{I_1(f)(x)}^q |g(x)| \, \dx \le C \left(\int_{\R^N} f(x)^p  \, \dx \right)^\frac{q}{p}, \quad \forall \, f \in \C_c(\R^N), f \ge 0.
\end{align}
For $N\geq 3$, $1< p \leq q <\infty$,  many authors provided various sufficient conditions on $g,h$ and $\gamma$ so that the following weighted inequality for the fractional integral holds:
\begin{align} \label{two_weight_ineq}
 \intRn  |I_\ga f(x)|^q |g(x)| \, \dx   \leq C \left( \intRn f(x) |h(x)|    \, \dx \right)^{\frac{q}{p}},   \quad \forall \, f \in \c1(\RN), f\ge 0.
\end{align}
For example, see \cite{Perez} ($p=q=2$ and $h \equiv 1$), \cite{Chan_Whed} ($p=q$),  \cite{Karman-Sawyer, LN,Sawyer1} ($p \leq q$).  In particular, for $\gamma=1$ and $h \equiv 1$, their results provide examples of $(p,q)$-Hardy potentials. In \cite[Theorem 1(A)]{Sawyer}, Sawyer-Wheeden have shown that, if there exist $s>1$ and $c>0$ such that
\begin{align}\label{SW}
|Q|^{\frac{\ga}{N}+\frac{1}{q}-\frac{1}{p}} \left( \frac{1}{|Q|} \int_{Q} |g(x)|^s \, \dx \right)^{\frac{1}{qs}} \left( \frac{1}{|Q|} \int_{Q} |h(x)|^{(1-\p)s} \, \dx \right)^{\frac{1}{\p s}} \le c,  \quad  \forall \, Q,
\end{align}
then  \eqref{two_weight_ineq} holds.
In particular, for $\ga =1$ and $h \equiv 1$ the above condition  reads as
\begin{align}\label{SW1}
     |Q|^{\frac{1}{\al(p,q)} -\frac{1}{s}} \left(\int_{Q} |g(x)|^s \, \dx  \right)^{\frac{1}{s}} \le c, \quad  \forall \, Q.  
 \end{align}
Thus $g$ satisfying \eqref{SW1} lies in $\pqR$.
Notice that, for $p=q=2$, \eqref{SW1} coincides with the Fefferman-Phong condition \eqref{FF1}. For recent developments concerning the weighted Sobolev inequalities and Feferman-Phong type conditions, we refer to \cite{FS,Tanaka} and the references therein.

\noi{\bf (ii) Bessel's pair:} Let $\Om=B_R(0)$ with $0<R\leq \infty$, and let $g,h$ be two positive, radial, $C^1$ functions on $\Om$. A pair $(g,h)$ is called  Bessel pair if $(g(r) r')' + h(r) r=0$ has a positive solution on $(0,R)$. In \cite{Ghoussoub2011}, the authors showed that, if $(r^{N-1}g,r^{N-1}h)$ is a  Bessel pair with $\int_0^R \frac{\dr}{r^{N-1}g(r)}=\infty$ and $\int_0^R r^{N-1} h(r) \dr<\infty$, then the following inequality holds:
\begin{align} \label{Gineq}
    \dis |g(x)| |u(x)|^2 \, \dx \le C \dis |h(x)| |\Gr u(x)|^2 \, \dx , \quad \forall \, u \in \c1(\Om).
\end{align}
For further improvements in this direction, we refer to \cite{LLZ1, LLZ} and the references therein.

\noi{\bf (iii) Maz'ya's capacity conditions:} Using $p$-capacity, Maz'ya has provided a necessary and sufficient condition (Theorem 8.5 of \cite{Mazya2}) on $g$ so that \eqref{p-qHardy} holds. Let us recall that, for $F \cset \R^N$, the p-capacity of $F$ with respect to $\Om$ is defined as
$$\text{Cap}_p(F,\Om)= \inf \left\{\int_{\Om} |\nabla u|^p : u \in \mathcal{C}^{1}_c(\Om), u \geq 1 \ \mbox{on} \ F \right\}\,.$$
For $1<p \leq q <\infty$, Maz'ya proved that  $g \in \mathcal{H}_{p,q}(\R^N)$ if and only if 
$$\|g\|_{\mathcal{H}_{p,q}}:=\displaystyle\sup_{F \cset \R^N} \left\{\frac{\displaystyle \int_{F}|g| }{[\text{Cap}_p(F)]^{\frac{q}{p}}}\right\}<\infty \,.$$
It is easy to see  that $\mathcal{H}_{p,q}(\Om) = \left\{g \in L^1_{loc}(\Om): \|g\|_{\mathcal{H}_{p,q}}<\infty \right\}; \ 1<p \leq q <\infty \, $
and $\|g\|_{\mathcal{H}_{p,q}}$ defines a Banach function space norm on $\mathcal{H}_{p,q}(\Om).$

\subsection{More admissible spaces of $(p,q)$-Hardy potentials}

The  results mentioned in the previous subsection  assumes $k=N$ and $q\ge p$ or assumes $k=N$ and $g$ is  radial. 
In this article, we allow the cases $0<q<p$ and $1\le k\le N.$ In these cases, depending on the values of $N,k,p,q$ and the geometry of $\Om,$ we  provide various classes of  function spaces that lie in $\pq$ mainly using two different techniques:
 $$(i)\; \text{symmetrization}, \quad (ii)\; \text{ polar decomposition}.$$
 
\noi $(i)$ The symmetrization method relies on the classical inequalities concerning symmetrization such as  P\'{o}lya-Szeg\"{o} inequality, Hardy-Littlewood inequality and the Muckenhoupt condition \cite{Muckenhoupt} for the one-dimensional weighted Hardy inequalities. 

\noi $(ii)$ The polar decomposition method is based on the use of the fundamental theorem of integration for various functions and  the H\"{o}lder's inequality for various conjugate pairs and conjugate triplets.

 One can also identify some admissible function spaces for  $(p,q)$-Hardy potentials using the embedding of  the Beppo-Levi space $\Dp$-the completion  of  $\c1(\Om)$ with respect to the the norm $\left(\int_{\Om} | \Gr u(x) |^p \, \dx \right)^{\frac{1}{p}}$ (if it is a well defined function space). For example, the Lorentz-Sobolev embedding and Moser-Trudinger embedding  of $\Dp$  provide certain Lebesgue spaces, Orlicz spaces, and  Lorentz spaces that lie in $\pq$ (Theorem \ref{Symmetrization}).  
 Notice that, if $\Om$ is  bounded in one direction, then $\Dp$ coincides with the classical Sobolev space $\wp.$   Unfortunately, for $N\le p$, $\dpR$ is not a function space. In fact, H\"{o}rmander-Lions in \cite{Hormander} showed that $\mathcal{D}^{1,2}_0(\R^2)$ contains objects that do not belong to even in the space of distributions. To ensure that $\Dp$ is a well defined function space,  we need to make some restrictions on $\Om.$ For $r \ge 0$, a open ball and a closed ball centred at $x$ with the  radius $r$ are denoted by  $B_r(x)$ and  $B_r[x]$ respectively.  Henceforth,  we make the following assumptions on the open set $\Om$:   
\begin{equation}\label{domain}\tag{\bf B}
    \begin{aligned}
  {\bf N>p:} &\quad\Om \text{ is any open set in } \R^N, \\
  {\bf N=p:}  &\quad \Om \subset \R^N \setminus B_a[x] \text{ for some } x \in \R^N, \text{ with } a>0, \\
  {\bf N<p:}  &\quad \Om \subset \R^N \setminus \{x\} \text{ for some } x \in \R^N.
\end{aligned}
\end{equation}
For $N \le p$ and for $\Om$ as given in \eqref{domain},  then we will show that $\Dp$ is always a well defined function space and it is continuously embedded in $L^p_{loc}(\Om)$ (Corollary \ref{function space} and see also \cite[Corollary 2.4]{ADS}).

\subsubsection{\bf{The $(p,q)$-Hardy potentials ($k=N$)}}
In this case, we have $\Om=\Om_1$ and $g=g_1.$ If $N>p$,  then using  \eqref{Sobolev} and the H\"{o}lder's inequality,  it is easy to  see that, for each $q\in (0,p^*],$  $$L^{\al(p,q)}(\Om)\subset \pq.$$
Furthermore, the classical Hardy-Sobolev inequality
\begin{align}\label{HS classic}
\int_{\Om} \frac{\abs{u(x)}^p}{\abs{x}^p}  \, \dx \le \left( \frac{p}{N-p} \right)^p \int_{\Om} \abs{\Gr u(x)}^p \, \dx, \quad \forall  \, u \in \c1(\Om)
\end{align} 
ensures that the Hardy potential $\frac{1}{|x|^p}$ belongs to $\mathcal{H}_{p,p}(\Om).$ Notice that, if $\Om$ contains the origin, then $\frac{1}{|x|^p} $ does not lie in any Lebesgue space. The inequality \eqref{HS classic} has been improved  by adding lower order radial weights to $\frac{1}{\abs{x}^p}$, for example, see \cite{Adimurthi, Brezis1, Tertikas, GM} and the references therein. Indeed, all these improved Hardy-Sobolev inequalities provide examples of radial weights in $\pp$. The improvements of \eqref{HS classic} involving the distance functions are available in \cite{LLZ1, LLZ, Lehrback}. Many authors are also interested to extend \eqref{HS classic} by considering more general class of weight functions in place of $\frac{1}{\abs{x}^p}$. The following version of Caffarelli-Kohn-Nirenberg  inequality \cite{CKN} extends  \eqref{HS classic} for $q \in [p,p^*]$:  \begin{align}\label{CKN classic}
   \intRn \frac{\abs{u(x)}^q}{\abs{x}^{\frac{N}{\al(p,q)}}}  \; \dx \leq C \left( \intRn \abs{\Gr u(x)}^p \, \dx \right)^{\frac{q}{p}}, \quad  \forall  \, u \in \c1(\RN).
\end{align}
 Thus $g(x)=|x|^{-\frac{N}{\al(p,q)}} \in \pqR$ for $q\in [p,p^*]$.
 \begin{remark} \label{nessq>=p}
 For $q\in [p,p^*]$, set $s=\frac{N}{\al(p,q)}$. Then \eqref{CKN classic} takes the following form:
  $$ \intRn  \frac{\abs{u(x)}^{P^*(s)}}{\abs{x}^{s}}  \; \dx \leq C \left( \intRn \abs{\Gr u(x)}^p \, \dx \right)^{\frac{N-s}{N-p}}, \quad  \forall  \, u \in \c1(\RN), $$
  with $s \in [0,p]$. The above inequality is also known as the classical Caffarelli-Kohn-Nirenberg  inequality. In \cite[Lemma 3.1]{GY}, the authors have shown that the conditions $q= P^*(s)$ and $s\in [0,p]$ are necessary for $|x|^{-s} \in \pqR$.
 \end{remark}
 
Observe that, $g=|x|^{-\frac{N}{\al(p,q)}}$ lies in the Lorentz space $L^{\al(p,q), \infty}(\Om)$ (see Example \ref{ex rearrangement}). In \cite{Visciglia}, for $p=2$ and $N> 2$, using the Lorentz-Sobolev embedding the authors have shown that $L^{\al(2,q), \infty}(\Om) \subset\pq $ for $q\in [2,2^*]$. For $N=p$ and $\Om=B_1(0)$, Edmunds-Triebel in \cite{ET} obtained an analogue of \eqref{HS classic}, namely:
\begin{align}\label{critical Hardy}
\int_{\Om} \frac{\abs{u(x)}^N}{(\abs{x}(\log(\frac{e}{|x|}))^N}  \, \dx \le \left(\frac{N}{N-1}\right)^N \int_{\Om} \abs{\Gr u(x)}^N \, \dx, \quad \forall  \, u \in \c1(\Om).    
\end{align}
Thus $(|x| (\log (\frac{e}{|x|})))^{-N} \in \mathcal{H}_{N,N}(B_1(0))$. It is not hard to verify that $g=(|x| (\log (\frac{e}{|x|}) ))^{-N}$ lies in the Lorentz-Zygmund space $L^{1, \infty;N}(B_1(0))$ (see Example \ref{ex rearrangement}). Indeed, following the same treatment as in \cite{Anoop1}, one can show that $L^{1, \infty;N}(\Om) \subset \mathcal{H}_{N,N}(\Om)$. Our first theorem improve all these results to general $N$, $p$ and also to a bigger range for $q$.
\begin{theorem}\label{Symmetrization}
Let $\Om$ be an open set in $\R^N.$ Let $ \ga=\frac{p}{p-q}$ for $q\in (0,p)$ and $\ga=\infty$ for $q\ge p.$  
\begin{enumerate}[(i)]
\item  Let $N>p.$ Then $$X:=L^{\al(p,q),\ga}(\Om) \subset \pq, \; \forall \, q \in [0,p^*]. $$ 
    \item Let $N=p$ and $\Om$ be bounded. Then 
    $$ X:= \left\{ \begin{array}{ll}
    L^{1,\ga;\frac{q}{p'}}(\Om), & \,  q\in (0,1) \cup [p, \infty);  \\
    L^{1, \ga; q-1}(\Om), &  \,  q\in [1,p),
\end{array}\right\} \subset \pq. $$
    \item Let $N < p$ and $\Om$ be bounded in one direction. Then  $$X:=L^1(\Om) \subset \pq, \; \forall \,  q\in [0,\infty).$$ 
 \end{enumerate}   
  Furthermore, for $g \in X$, there exists $C=C(N,p,q)>0$ so that 
 \begin{align}\label{eqn:Lorentz}
 \int_{\Om} |g(x)| |u(x)|^q \, \dx  \leq C\norm{g}_{X} \left( \int_{\Om} | \Gr u(x) |^p \, \dx \right)^{\frac{q}{p}},   \quad \forall  \, u \in \c1(\Om).
\end{align}

\end{theorem}

Our proof for the above theorem is based on the embeddings of $\Dp$ into various function spaces with respect to the values of $N$ and $p$. For $N>p$, in \cite{ONeil} O'Neil proved that $\Dp$ is embedded in the Lorentz space $L^{p^*,p}(\RN).$  For $N=p$ and  $\Om$ bounded,  $\Dp$  is embedded in the  Lorentz-Zygmund space $L^{\infty, N ; -1}(\Om)$, proved independently  by  Hansson \cite{Hansson} ($N=p=2$), Brezis-Wainger \cite[Lemma 1]{BW} (for $N=p$). For $N<p$ and  $\Om$ bounded in one direction, $\Dp$ coincides with $W^{1,p}_0(\Om)$ and hence embedded into the Lebesgue space $L^\infty(\Om)$. In the appendix, we give simple alternate proofs for all these embeddings (Theorem \ref{embedding}) using the Muckenhoupt condition for the one-dimensional weighted Hardy inequalities and certain classical inequalities such as P\'{o}lya-Szeg\"{o} inequality, Hardy-Littlewood inequality. We also prove that,  if $g$ is radial, radially decreasing, then for $g$ to be a $(p,q)$-Hardy potential, it is necessary that $g$ lies in the spaces given in the above theorem (see Proposition \ref{necessarycond}).

Next we produce another class of $(p,q)$-Hardy potentials on certain symmetric open sets  via the polar decomposition method. 

\begin{definition}[\bf The sectorial sets]
Let $1\le k\le N$ and $S$ be an open subset of $\S^{k-1}$ and $a,b \in[0,\infty]$ with $a<b$. Then consider the open set \begin{equation}\label{def:sector}
    \Om_{a,b,S}= int\left(\left\{x\in \Rk: a \leq |x| < b, 
\frac{x}{|x|}\in S \ \text{if} \ x \neq 0 \right\}\right),
\end{equation} where $int(A)$ denotes the interior of a set $A$. 
\end{definition}

 Notice that, $0\in \Om_{a,b,S},$ only if  $a=0, S=\mathbb{S}^{k-1}$. If $S_1=\{x=(x_1,...,x_k) \in \mathbb{S}^{k-1}:x_1>0\}$, then   $\Om_{0,\infty,S_1}$ becomes the half space $\R^k_+=\left\{x\in \R^k:x_1>0\right\}.$  Next we associate a radial function to a $L^1_{loc}(\Om_{a,b,S})$ function via the notion of radial majorant.
 
\begin{definition}[\bf{The radial majorant}]
For $f\in L^1_{loc}(\Om_{a,b,S}),$ we define the radial majorant of $f$ as below: 
\begin{equation}\label{def:majorant}
    \f(r) = \esssup \{ |f(r \om)|: \om \in S \}, \ \  r\in (a,b),
\end{equation}
where the essential supremum is taken with respect to the $(k-1)$-dimensional surface measure.
\end{definition} 

\noi Notice that, $\f(r)$ is finite a.e. in $\Om_{a,b,S}$ (\cite[Theorem 2.49]{Folland}) and  for a radial function $f$, $f(x)=\f(|x|).$ Moreover,  every function defined on $\Om_{a,b,S}$ is  dominated by it's radial majorant. 

In \cite{ADS}, for $\Om_{1,\infty,\S^{N-1}}$ ($=B_1^c$-the exterior of the unit ball centred at the origin) and  for $q=p$, authors have considered class of weight functions that are dominated by radial functions. They have  shown that, if $g$ is dominated by a radial function $w$ and $w\in L^1((1, \infty),r^{p-1})$ (\cite[Theorem 1.1]{ADS}), then $g \in \mathcal{H}_{p,p}(B_1^c)$. Observe that, the radial majorant $\tilde{g}\le w$ and hence the same result is true, if $\tilde{g}\in L^1((1, \infty),r^{p-1}).$  In this article, we extend this result for $q\in[0,p^*]$ and for the general sectorial sets.

\begin{theorem}\label{weighted Lebesgue}
For $S \subset \S^{N-1}$ and $a,b \in [0, \infty]$ with $a<b$, let $\Om=\Om_{a,b,S}$  be the sectorial set as given in \eqref{def:sector} and let $\g$ be the radial majorant of $g\in L^1_{loc}(\Om).$   For $q\in (0,\infty)$, let 
$$X:= \left\{ \begin{array}{ll}
    L^1((a, b), r^{\frac{N}{\al(p,q)} - 1}), & \, p\ne N;  \\
    L^1\left((a, b), r^{N-1} \left( \log \left( \frac{r}{a} \right) \right)^{\frac{q}{N^{'}}}\right), &  \,p= N.
\end{array}\right.$$

\begin{enumerate}[(i)]
    \item  $N>p$: For $q\in (0,p^*]$, let $\g  \in X $ and in addition $\g$ be strictly decreasing for $q\in \left[P^*(1), p^* \right]$. Then $g \in \pq.$
    \item $N=p:$ For $q\in (0,p]$, let $\g\in X$ and $a>0$. Then  $g \in \pq.$
    \item $N<p:$ For $q\in (0,p]$, let $\g\in X$ and $0\notin \Om.$ 
    Then  $g \in \mathcal{H}_{p,q}(\Om).$
\end{enumerate}
Furthermore, there exists $C=C(N,p,q)>0$ so that 
    \begin{align*}
 \dis |g(x)| |u(x)|^q \, \dx  \le C\norm{\g}_{X} \left( \dis | \Gr u(x) |^p \, \dx \right)^{\frac{q}{p}},   \quad \forall  \, u \in  \c1(\Om). 
\end{align*}
\end{theorem}

As the immediate consequences of the above theorem, for $\Om$ as given in Theorem \ref{weighted Lebesgue}, we have the well definedness of the Beppo-Levi space $\Dp$ for $N \le p$. Recall, $B_r[x]$ is the closed ball centred at $x$ with the radius $r$.

\begin{corollary}\label{function space}
Let $\Om = \RN \setminus B_a[0]$ with $a>0$ if $N=p$ and $a=0$ if $N<p$. Then $\Dp$ is continuously embedded in $W_{loc}^{1,p}(\Om)$, i.e.,  for every compact set $K$ in $\Om$, there exists $C=C(K,p) >0$ such that
$$  \int_{K} \left( |u(x)|^p + |\Gr u(x)|^p \right) \, \dx \le C \dis |\Gr u(x)|^p \, \dx, \quad \forall \, u \in \Dp.$$
\end{corollary}

\subsubsection{\bf{The cylindrical $(p,q)$-Hardy potentials ($1 \leq k<N$)}}

The weight functions provided by Theorem \ref{Symmetrization} and Theorem \ref{weighted Lebesgue} do not exhaust the entire $\pq$. In \cite{BT}, for $N>p,$  $q \in [p,p^*]$ and $\frac{N}{\al(p,q)}< k\leq N$ (equivalently, $q \in  (P^*(k),p^*]$ if $k\le p$, and  $q \in [p,p^*]$ if $k>p$) Badiale-Tarantello obtained the following cylindrical version of the C-K-N inequality for $\Om=\R^N$ and $k\ge 2$:
\begin{align}\label{p-qHardy2}
      \int_\Om \frac{|u(x)|^q}{\abs{y}^{\frac{N}{\al(p,q)}}}  \, \dx   \le C \left( \int_\Om | \Gr u(x) |^p \, \dx \right)^{\frac{q}{p}},   \quad \forall  \, u \in \c1(\Om).
\end{align}
 In \cite{BS,BT,Sandeep}, using the above inequality, it has been proved that, if $r^{\frac{N}{\al(p,q)}}\phi(r) \in L^{\infty}((0,\infty))$, then the cylindrical weights $g(x)=\phi(|y|) \in \mathcal{H}_{p,q}(\R^N)$. In this article, we consider  more general class of  domains and the weight functions of the form given in \eqref{products} and also allow $k=1$. 
For brevity, we only consider the case $N>p$. First we extend \eqref{p-qHardy2} for more general sectorial sets.

\begin{theorem}\label{cylin C-K-N}
Let $p \in (1,N)$ and $1 \le k\le N.$ For $S \subset \S^{k-1}$ and $a,b \in [0, \infty]$ with $a<b$, let $\Om=\Om_{a,b,S}\times \R^{N-k}.$
\begin{enumerate}[(i)]
    \item Then $|y|^{-\frac{N}{\al(p,q)}} \in \pq$ for
    $q \in \left\{ \begin{array}{ll}
    [p,p^*],   & \ \  k >p;  \\
    (P^*(k) ,p^*],  & \ \ k \leq p. 
\end{array}\right. $
    \item  If $0\not\in \Om_1$ and $k<p$, then $|y|^{-\frac{N}{\al(p,q)}} \in \pq$ for $q \in [p, P^*(k)]$.
  \end{enumerate}
\end{theorem}

\begin{remark}\label{re cylin CKN}
$(i)$ If $a=0$ and $b = \infty$, then for $|y|^{-s}$ to be in $\pq,$ it is necessary that  $s=\frac{N}{\al(p,q)}$, one can see this by considering the scaling of a function. On the other hand, if $a>0$ then Theorem \ref{cylin C-K-N}  holds for  $|y|^{-s}$ with $s \geq \frac{N}{\al(p,q)}$, and if $b< \infty$, then Theorem \ref{cylin C-K-N}  holds for  $|y|^{-s}$ with $s \leq \frac{N}{\al(p,q)}$. Since $|x|^{-s}\le |y|^{-s}$, the restriction  $ p\le q\le p^*$ is also necessary for \eqref{p-qHardy2} (see also Remark \ref{nessq>=p}).

\noi $(ii)$ In \cite[Theorem 1.1]{Lehrback}, authors extended \eqref{p-qHardy2} by replacing $|y|$ (which is the distance of $x$ from $\R^{N-k}$) with distance function $\delta_E$ from general closed set $E$ in $\R^N$ and also allow the case $k=1$. More precisely, for $q \in [p,p^*]$, they have established 
    $$\int_{\R^N} \frac{|u(x)|^q}{|\delta_E(x)|^{\frac{N}{\al(p,q)}}} \ \dx \leq C \left(\int_{\R^N} |\nabla u|^p \ \dx \right)^{\frac{q}{p}}, \quad \forall \, u \in \c1(\R^N),$$
if and only if, the Assouad dimension of $E$ is strictly less than $\frac{Nq}{p^*}$. In particular, for $E=\R^{N-k}$, the Assouad dimension of $E$ is $N-k$ and $\de_E(x)=|y|$ and hence, \eqref{p-qHardy2} holds if and only if $N-k < \frac{Nq}{p^*}$.  Thus for \eqref{p-qHardy2} to hold, we must have $q>P^*(k)$  as given in part $(i)$ of the above theorem. In part $(ii)$, we addresses the complementary case:  $q \in [p,P^*(k)]$ for $k<p$ on a sectorial set with a hole at the origin.
    
    \noi $(iii)$   For $q \in (0,P^*(k)]$,  $|y|^{-\frac{N}{\al(p,q)}}$ is not locally integrable on any open set in $\R^k$ that contains the  origin.  On the other hand, if $\Om_1$ does not contain the origin and $k<p,$ then the above theorem ensures that $|y|^{-\frac{N}{\al(p,q)}}\in \pq$ for $q\in [p,p^*]$. In particular, by taking $k=1, S=\{1\},a=0,$ and $b=\infty$, we obtain \eqref{p-qHardy2}  for $\Om=\R^N_+$ and for $q\in [p,p^*]$.
\end{remark}

The following corollary is immediate from the above theorem. 
\begin{corollary}\label{gen cylin C-K-N}
Let $\Om, q$ be as given in Theorem \ref{cylin C-K-N}. Let $g(x)=g_1(y)$ be such that $\g_1 \in L^\infty((a,b), r^{\frac{N}{\al(p,q)}})$. Then the same conclusions of  Theorem \ref{cylin C-K-N} hold for $g$ in place of $|y|^{-\frac{N}{\al(p,q)}}.$
\end{corollary}

Next, we consider the case in which both  $g_1$ and $g_2$ are in certain  Lorentz spaces. Indeed, the product of any two functions from the Lorentz spaces need not be a $(p,q)$-Hardy potential (see Example \ref{non-compatible}). Here we provide a two-parameter family of compatible pairs of Lorentz spaces so that the product of functions from these Lorentz spaces always give rise to a  $(p,q)$-Hardy potential.
\begin{theorem}\label{Lorentz product}
Let $p \in (1,N)$ and let $t \in [0,1].$ Let $k>p$ if $t>0$, and $N-k>p$  if $t<1$.  Let $\Om$ and $g$ be as given in \eqref{products}.
  \begin{enumerate}[(i)]
      \item For  $s,t\in [0,1]$ with $st<1$, let  $(g_1,g_2) \in X_1\times X_2:= L^{\frac{k}{(k-p)st+p}, \frac{1}{st}}(\Om_1) \times L^{\frac{1}{st}}(\Om_2).$ Then  $g \in \pq$ for $q=(1-st)p.$ 
      \item For  $s,t\in [0,1],$ let $(g_1,g_2) \in X_1\times X_2:= L^{\frac{k}{(1-s)tp}, \infty}(\Om_1) \times L^{\frac{N-k}{(1-t)p}, \infty}(\Om_2).$ Then   $g \in \pq$ for  $q = (1-st)p+stp^*.$ 
      \end{enumerate}
Furthermore,  
\begin{align*}
    \dis |g(x)| |u(x)|^q \, \dx \le C \norm{g_1}_{X_1} \norm{g_2}_{X_2} \left(  \dis |\Gr u(x)|^p \, \dx \right)^{\frac{q}{p}}, \quad \forall  \, u \in \c1(\Om),
\end{align*}
for some $C=C(N,k,p,q)>0$.
\end{theorem}

In the following theorem,  we  consider the compatible pairs of $(g_1,g_2)$ from the weighted Lebesgue spaces. For $q\in (0,p^*]$,  we set $$\be(p,q):= \frac{{p^*}-{N'}p}{{p^*}-{N'}q}.$$

\begin{theorem}\label{weighted Lebesgue product}
Let $p \in (1,N)$ and $k<N$. For $S \subset \S^{k-1}$ and $a,b \in (0, \infty]$ with $a<b$, let $\Om= \Om_{a,b,S} \times \RNk$ and $g$ be as given in \eqref{products}. If $k\ne p$ and 
   \begin{align*}
    (\g_1,\g_2)  \in X_1\times X_2:= \left\{\begin{array}{ll}
    L^1((a,b),r^{\frac{(p-k)q}{p}+k-1}) \times  L^{\frac{p}{p-q}}((0, \infty),r^{N-k-1}), \; & \ \ q \in (0,p); \\
    L^{\be(p,q)}((a,b), r^{p-1})\times L^{\frac{\be(p,q)}{\be(p,q)-1}}((0,\infty)), \; & \ \ q \in [p, P^*(1)); \\
    L^{\infty}((a,b))\times L^{\frac{\al(p,q)}{N}} ((0,\infty)), \; & \ \ q \in \left[ P^*(1),p^* \right],
 \end{array} \right. 
 \end{align*}
then $g \in \pq$. Furthermore, 
\begin{align*}
       \dis |g(x)| |u(x)|^q \, \dx \le C \norm{\g_1}_{X_1} \norm{\g_2}_{X_2} \left(  \dis |\Gr u(x)|^p \, \dx \right)^{\frac{q}{p}}, \quad \forall  \, u \in \c1(\Om),
   \end{align*}
where $C=C(N,k,p,q)>0$.
\end{theorem}

Having obtained a large class of  cylindrical and non-cylindrical  $(p,q)$-Hardy potentials, next we consider the existence of solution for the Euler-Cauchy equation associated to  \eqref{p-qHardy}.
  
\subsection{The existence of solution}
For $g\in \pq$ with $g\ge 0$, let $B_q(g)$ be the best constant in \eqref{p-qHardy}. Then
\begin{equation} \label{best}
\begin{aligned}
\frac{1}{B_q(g)}=\inf \left\{ \frac{\int_{\Om} |\nabla u|^p}{\left(\int_{\Om} g|u|^q\right)^{\frac{p}{q}}}: u \in \Dp \setminus \{0\} \right\} = \inf \left\{ \int_{\Om} |\nabla u|^p: u \in N_g\right\}, 
\end{aligned}
\end{equation}
where  $ N_g = \left\{ u \in \Dp : \int_{\Om} g |u|^q  = 1 \right\}$. If $\frac{1}{B_q(g)}$  is attained for some  $u\in N_g$, then one can verify that for $q>1$, $u$  satisfies the following equation:
$$\int_{\Om} |\nabla u|^{p-2} \nabla u \cdot \nabla v = \frac{1}{B_q(g)} \int_{\Om} g |u|^{q-2}uv, \quad \forall  \,  v  \in \Dp.$$
In other words, for $\la= \frac{1}{B_q(g)},$ $u$ solves the following nonlinear partial differential equation weakly:
\begin{align}\label{eqn:evp}
         -\De_p u = \la g(x) \abs{u}^{q-2}u, \quad u \in \Dp,
\end{align}
where $\Delta_p$ is the $p$-Laplace operator defined as $\Delta_p u =  {\rm div} (|\Gr u|^{p-2} \Gr u)$. Observe that, for  $q \neq p$,  $v=(\la B_q(g))^{- \frac{1}{q-p}} u$ solves the above equation for any $\la>0.$ For $q=p$, \eqref{eqn:evp} is a nonlinear eigenvalue problem and a non-zero solution exists only for certain $\la$ that are precisely the eigenvalues of \eqref{eqn:evp}.

The above partial differential equation appears in many important problems in mathematics as well as in physics. For example, radial $g$ in the Matukuma's models for the dynamics of globular cluster of stars \cite{Li, NS, Y1},  cylindrical potentials in the study of  dynamics of galaxies \cite{Bertin, Ciotti},  scalar curvature problem \cite{Li1},   the weighted eigenvalue problems \eqref{eqn:evp} for $q=p$. One of the  sufficient conditions that ensure the best constant $B_q(g)$ is attained in $\Dp$ is the compactness of the following map:
 $$G_q(u) :=\int_\Om g|u|^q, \ u \in \Dp.$$  Indeed, for $g\in\pq$, from \eqref{p-qHardy} it is clear that $G_q$  is continuous.

For $q=p$, in the context of studying weighted eigenvalue problems, many authors considered $g$ in various Lebesgue and Lorentz spaces so that the map $G_q$ is compact. For example, for $g \in L^\ga(\Om)$, see \cite{MM} ($N>p=2,\ga > \frac{N}{2}$), \cite{Allegretto} ($N>p=2,\ga=\frac{N}{2}$), \cite{Allegretto-Huang, Szulkin} ($N>p, \ga=\frac{N}{p}$). For $g \in L^{\frac{N}{p}, \ga}(\Om)$, see \cite{Visciglia} ($N>p=2$,$\ga<\infty$), \cite{AMM} ($N>p=2$, $\ga=\infty$),  \cite{Anoop1} $(N=p=2)$. In \cite{ADS}, the authors obtained the compactness of $G_p$ for $g$ dominated by a certain radial function. For $q\neq p$, there are few results where the compactness of $G_q$ is proved. For example, for $p \in (1,N)$ and $q \in (0,p^*)$, $g \in L^{\infty}(\Om) \cap L^{\al(p,q)}(\Om)$ \cite{Yu}, for $q \in [2, 2^*),$ $g \in  L^{\al(2,q),\ga}(\Om)$ with $1 \le \ga < \infty$ \cite{Visciglia} and $g$ in the closure of $\cc(\Om)$ in $L^{\al(2,q), \infty}(\Om)$ \cite{AMM}. In this article, we state certain general assumptions on $g$ that ensures the compactness of $G_q$ and unify all the above compactness results. 

For $i=1,2$, let $X_i(\Om_i)$ be Banach (function) space containing $\cc(\Om_i)$.  We define $\F_{X_i}:=\overline{\cc(\Om_i)}^{X_i}.$ 
\begin{theorem} \label{cpct1}
Let $\Om$ and $g$ be as given in \eqref{products} and \eqref{domain} and let $g \ge 0$. For $i=1,2$, let $g_i \in \F_{X_i}$ and the following inequality holds: \begin{equation} \label{normineq}
     \int_{\Om} g(x) |u(x)|^q \, \dx \le C \norm{g_1}_{X_1} \norm{g_2}_{X_2} \left( \int_\Om |\Gr u(x)|^p \, \dx \right)^{\frac{q}{p}}, \quad \forall  \, u \in \c1(\Om).
  \end{equation}
Then the map $G_q=\int_{\Om} g|u|^q$ is compact on $\Dp$ for $q \in (0, \de)$, where $\de=p^*$ (if $N>p$), and 
$\de=\infty$ (if $N \le p).$ Moreover, for $q \in (1,\delta)$, \eqref{eqn:evp} admits a non-negative solution in $\Dp$.
\end{theorem}

\begin{remark} $(i)$ For $q<p$, $\cc(\Om_i)$ is dense in the function spaces $X_i$ as considered in Theorem \ref{Symmetrization} and Theorem \ref{Lorentz product}. Thus by the above theorem, for $g_i \in X_i$, the map $G_q$ is compact. For $q \ge p$, the map $G_q$ is compact for $g_i$ in $\F_{X_i}$-a proper closed subspace of the respective space $X_i$. As a consequence, Theorem \ref{cpct1} together with Theorem \ref{Symmetrization} extends the compactness results of \cite{AMM, Visciglia} to $q \in (0,p^*)$, \cite{Anoop1} to $q \in (0, \infty)$.

\noi $(ii)$ Since $\cc((a,b))$ is dense in $X$ given in Theorem \ref{weighted Lebesgue}, for $\g \in X$ using Theorem \ref{cpct1} the map $G_q$ is compact.  For $q \in (0, P^*(1))$, $\cc((a,b))$ is dense in  $X_i$ given in Theorem \ref{weighted Lebesgue product}, and hence for $(\g_1, \g_2) \in X_1 \times X_2$,  $G_q$ is compact.

\noi $(iii)$ Some of the non-compact cases for $p=q$ are discussed in  \cite{Anoop_Ujj} using the variant of concentration compactness lemma and Maz'ya's capacity conditions.
\end{remark}

The rest of this article is organized as follows. In Section 2, we briefly discuss symmetrization and recall the Lorentz and Lorentz-Zygmund spaces. In Section 3, we prove some important properties of $\pq$ that are required in subsequent sections. Section 4, Section 5, and Section 6 contain the proofs of theorems \ref{Symmetrization}-\ref{cpct1}. In Section 7, we discuss some examples and the necessary conditions. In Appendix, we prove some results on Lorentz and Lorentz-Zygmund spaces and present alternative proofs of certain classical embeddings. 

%% file: Preliminaries.tex
\section{Preliminaries}\label{preliminary}
 In this section, we briefly describe the symmetrization and the one dimensional decreasing rearrangements. Using this, we define the Lorentz and Lorentz-Zygmund spaces and list some of their properties.
 
 Firstly, we list some of the notations and conventions we used in this article:
 \begin{itemize}
    \item $\frac{1}{0}= \infty$.
    \item $\om_N := \frac{\pi^{\frac{N}{2}}}{\Gamma(\frac{N}{2}+1)}$ is the measure of a unit ball in $\RN$. 
    \item For $q\in [1,\infty],\, q'$ denote the conjugate of $q,$ i.e., $\frac1q+\frac{1}{q'}=1.$
    \item For $a,b,c\in [1,\infty]$ we say $(a,b,c)$ is a conjugate triple, if $\frac1a+\frac{1}{b}+\frac{1}{c}=1$, and we say $(a,b)$ is a conjugate pair if $(a,b,\infty)$ is a conjugate triple.
    \item If $a\in[0,1]$ and $b\in [1,\infty)$, then $(\frac{1}{a},\frac{1}{1-a})$ and $(b,\frac{b}{b-1})$ are conjugate pairs.  If $a,b\in[0,1]$, then $(\frac{1}{a}, \frac{1}{b}, \frac{1}{1-a-b})$ is a conjugate triple.
    \item For  $u:\R^N \mapsto \R$ and $z \in \RNk$, the $z$-section of $u$  is denoted by $u_z$ i.e., $u_z(y)=u(y,z)$ $\forall  \, y \in \R^k$. Similarly, the $y$-section of $u$  is denoted by $u_y$ i.e., $u_y(z)=u(y,z), \forall  \, z \in \R^{N-k}$. 
    \item For $u \in C^1(\R^N)$ and $z \in \RNk$,  $\Gr_{y} u(y,z) := \Gr_y u_z(y)=\left(\frac{\pa u_z}{\pa x_1},\frac{\pa u_z}{\pa x_2},...,\frac{\pa u_z}{\pa x_k}\right).$
    \end{itemize}

\subsection{Symmetrization}

Let $\Om \subset \RN$ be an open set and  $\M(\Om)$ be the set of all extended real valued Lebesgue measurable functions that are finite a.e. in $\Om$. Given a function $f \in \M(\Om)$ and for $s > 0, $ we define $E_{f}(s) = \{ x \in \Om : |f(x)| > s \}.$  The \textit{distribution function}  $\mathcal{\mu}_{f}$ of $f$ is defined as
$
\mu_{f}(s) = |E_{f}(s)|,
$
where $|\cdot|$ denotes the Lebesgue measure in $\RN$. We define the \textit{one dimensional decreasing rearrangement} $ f^{*} $ of $f$ as
\begin{equation*}
		f^*(t) = \inf \{ s > 0 : \mu_{f}(s) < t \}, \; \mbox{ for } t > 0. 
\end{equation*}
The map $f \mapsto f^*$ is not sub-additive. However, we obtain a sub-additive function from $f^*,$ namely the maximal function $f^{**}$ of $f^*$, defined by 
\begin{equation*}\label{Maximal 1}
 f^{**}(t)=\frac{1}{t}\int_0^t f^*(\tau)\, {\rm d}\tau, \quad t>0.
 \end{equation*}

Next, we explicitly compute the rearrangement of certain class of functions. 
\begin{remark}\label{rearrangement for radial}
Let $g$ be a non-negative strictly decreasing function on $\R^+ \cup \{0\}$. Let $f(x):= g(\abs{x})$ for $x \in \R^N.$ Then for $s \in Range(g)$, 
\begin{align*}
    E_f(s) = \left\{ x \in \RN : g(|x|) > s \right\}= \left\{ x \in \RN : |x| < g^{-1}(s) \right\},
\end{align*}
and $\mu_f(s) = \om_N (g^{-1}(s))^N.$ Hence for $t>0,$ 
\begin{align*}
    f^*(t)  = \inf  \left\{ s>0 : \mu_f(s) \le t \right\} & = \inf  \left\{ s \in (\inf g, \sup g) : \mu_f(s) \le t \right\} \\ & \leq \inf \left\{ s \in Range(g) : \om_N (g^{-1}(s))^N \leq t \right\} \\
    &=\inf \left\{ s \in Range(g) : s \ge g \left( \om_N^{-\frac{1}{N}} t^{\frac{1}{N}} \right) \right\} = g \left( \om_N^{-\frac{1}{N}} t^{\frac{1}{N}} \right).
\end{align*}
Further, notice that if $g$ is onto, then we have $f^*(t)=g ( \om_N^{-\frac{1}{N}} t^{\frac{1}{N}})$. 
\end{remark}

Now we state two important inequalities concerning symmetrization. For more details we refer to the books \cite{HLP, PS}.
\begin{proposition}\label{HL and PS}
Let  $N \ge 2$.
\begin{enumerate}[(a)]
\item \textbf{Hardy-Littlewood inequality}: Let $f$ and $g$ be nonnegative measurable functions. Then
 $$\dis f(x)g(x) \, \dx \le \int_0^{|\Om|}f^*(t) g^*(t)\, \dt.$$
\item \textbf{P\'{o}lya-Szeg\"{o} inequality}: 
Let $u \in  \D^{1,p}_0(\RN)$. Then  
$$ \displaystyle N^p \om_N^{\frac{p}{N}} \int_{0}^{\infty} s^{p - \frac{p}{N}} | {u^*}^{\prime}(s)|^p \, ds  \le \int_{\RN} | \Gr u(x)|^p \, \dx.$$
\end{enumerate}	
\end{proposition}

The following inequality  is due to Maz'{j}a \cite[Lemma 1, Pg-49]{Mazja}.

\begin{proposition}\label{Mazya1}
Let $q \ge 1$. Then, for any measurable function $f: \R^N \mapsto \R $ the following inequality holds
\begin{align*}
    \int_0^{\infty} f^*(t)^q \, {\rm d}t^q \le \left( \int_0^{\infty} f^*(t) \, \dt \right)^q.
\end{align*}
\end{proposition}

\subsection{The Lorentz and Lorentz-Zygmund spaces}
The Lorentz spaces are two parameter family of function spaces introduced by Lorentz in \cite{Lorentz} that refine the classical Lebesgue spaces. For more details on the Lorentz spaces, we refer to \cite{Adams,EdEv}.

Let $\Om$ be an open set in $\RN$ and $f \in \M(\Om)$. For $(p,q) \in (0,\infty) \times (0,\infty]$ we consider the following quantity:
 \begin{align}
 |f|_{p,q} := \norm{t^{\frac{1}{p}-\frac{1}{q}} f^{*} (t)}_{{L^q((0,\infty))}}
 =\left\{\begin{array}{ll}
 \left(\displaystyle\int_0^\infty \left(t^{\frac{1}{p}-\frac{1}{q}} {f^{*}(t)}\right)^q \, \dt \right)^{\frac{1}{q}}, \quad & q < \infty; \\ 
 \displaystyle\sup_{t>0}t^{\frac{1}{p}}f^{*}(t),\quad & q=\infty.
 \end{array} 
 \right. 
 \end{align}
The Lorentz space $L^{p,q}(\Om)$ is defined as
 \[ L^{p,q}(\Om) := \left \{ f \in \M(\Om): \, |f|_{p,q} < \infty \right \},\]
where $ |f|_{p,q}$ is  a complete quasi norm on $L^{p,q}(\Om).$ For $(p,q) \in (1, \infty] \times (0, \infty]$, 
$$ \norm{f}_{p,q} := \norm{t^{\frac{1}{p}-\frac{1}{q}} f^{**}(t)}_{{L^q((0,|\Om|))}} $$
is a norm on $L^{p,q}(\Om)$ and it is equivalent to  $ |f|_{p,q}$ \cite[Lemma 3.4.6]{EdEv}. Note that $L^{p,p}(\Om) = L^p(\Om)$ for $p \in (0, \infty)$ and $ L^{p,\infty}(\Om)$ coincides with the weak-$L^p$ space $:= \{ f \in \M(\Om): \underset{s> 0}{\sup} \, s(\alpha_f(s))^{\frac{1}{p}}< \infty \}.$

\begin{remark}\label{equivalent}
For $0 < s < p$ and $f \in L^{p, \infty}(\RN)$, we consider the following quantity:  \begin{align*}
    ||| f |||_{p, \infty} := \sup_{\{E \subset \RN,|E|< \infty \}} |E|^{\frac{1}{p}-\frac{1}{s}} \left( \int_{E}|f(x)|^s \, \dx \right)^{\frac{1}{s}}.
\end{align*}
Then $\norm{f}_{p, \infty} \le |||f|||_{p, \infty} \le \left( \frac{p}{p-s} \right)^{\frac{1}{s}} \norm{f}_{p, \infty}$ (see \cite[Theorem 5.18]{Castillo}). 
\end{remark}

In the following  proposition we list some properties of the Lorentz spaces.
\begin{proposition}\label{Lorentz properties}
 Let $p,q,\tilde{p},\tilde{q} \in [1, \infty]$.
 \begin{enumerate}[(i)]
    \item  For $\al>0$, $\norm{\abs{f}^{\alpha}}_{\frac{p}{\alpha}, \frac{q}{\alpha}} = \norm{f}^{\alpha}_{p,q}.$ 
    \item Generalized H\"{o}lder inequality: Let $f \in L^{p_1, q_1}(\Om)$ and $g \in L^{p_2, q_2}(\Om)$, where  $(p_i, q_i) \in (1, \infty) \times [1, \infty]$ for $i = 1,2$. If $(p,q)$ be such that $\frac{1}{p} = \frac{1}{p_1} + \frac{1}{p_2}$ and 
          $\frac{1}{q} = \frac{1}{q_1} + \frac{1}{q_2},$ then 
          \begin{align*}
                \norm{fg}_{p,q} \leq C \norm{f}_{p_1,q_1} \norm{g}_{p_2,q_2},
           \end{align*}
          where  $C = C(p) > 0$ is a constant such that $C = 1,$ if $p=1$ and $C = \p,$ if $p > 1$.    
    
    \item  If $q \leq \tilde{q},$ then  $L^{p,q}(\Om) \hookrightarrow L^{p, \tilde{q}}(\Om)$, i.e.,  there exists a constant $C > 0$ such that
           \begin{align}\label{embed1}
                \norm{f}_{p,\tilde{q}} \leq C \norm{f}_{p,q}, \quad \forall  \,  \, f \in  L^{p,q}(\Om).    
            \end{align}
    \item If $\tilde{p} < p,$ then $L^{p,q}(\Om) \hookrightarrow L_{loc}^{\tilde{p}, \tilde{q}}(\Om)$. 
\end{enumerate}
\end{proposition}

\begin{proof} Proof of (i) directly follows using the definition of the Lorentz space. Proof of (ii) follows using \cite[Theorem 4.5]{Hunt}. For the proof of (iii) and (iv), see \cite[Proposition 3.4.3 and Proposition 3.4.4]{EdEv}.
\end{proof}

The Lorentz-Zygmund spaces are three parameter family of function spaces that refine the Lorentz spaces. For more information on Lorentz-Zygmund spaces, we refer to \cite{ColinRudnik, ET}. Let $\Om \subset \R^N$ be a bounded open set and let $l_1(t) = \log ( \frac{e |\Om|}{t} )$. Given a function $f \in \M(\Om)$ and for $(p,q, \al) \in (0,\infty] \times (0,\infty] \times \R$, we consider the following quantity:
\begin{align*} 
       |f|_{p,q, \al} := \norm{t^{\frac{1}{p}-\frac{1}{q}} {l_1(t)}^{\al} f^{*}(t)}_{{L^q((0,|\Om|))}} 
       =\left\{\begin{array}{ll}
                              \left(\displaystyle\int_0^{|\Om|} \left(t^{\frac{1}{p}-\frac{1}{q}} {l_1(t)}^{\alpha} {f^{*}(t)} \right)^q \, \dt \right)^{\frac{1}{q}}, & 0< q < \infty;\\ 
                              \displaystyle\sup_{0 < t < |\Om|} t^{\frac{1}{p}} {l_1(t)}^{\alpha} {f^{*}(t)}, &\; q=\infty.
                            \end{array}\right.
 \end{align*}
Then the Lorentz-Zygmund space $L^{p,q;\al}(\Om)$ is defined as
\[ L^{p,q;\al}(\Om) := \left \{ f\in \M(\Om): \,   |f|_{p,q,\al}<\infty \right \},\]
where $ |f|_{(p,q;\al)}$ is the quasi norm on $L^{p,q;\al}(\Om)$, and for $(p,q,\al)\in (1, \infty) \times [1, \infty] \times \R$,   
\begin{align}\label{equivallence norm}
\norm{f}_{p,q, \al} = \norm{t^{\frac{1}{p}-\frac{1}{q}} {l_1(t)}^{\al} f^{**}(t)}_{{L^q((0,|\Om|))}}
 \end{align}
is a norm in $L^{p,q;\al}(\Om)$ equivalent to $\abs{f}_{p,q, \al}$ \cite[Corollary 8.2]{ColinRudnik}. In the following proposition we list some important properties of the Lorentz-Zygmund spaces.

\begin{proposition}\label{LZ prop}
Let $p,q,\tilde{q} \in [1, \infty]$ and $\al, \beta \in (-\infty, \infty).$ 
\begin{enumerate}[(i)]
    \item Let $p,q \in (1, \infty], \al \in \R$, and $\ga >0$. Then there exists $C>0$ such that $\norm{|f|^{\ga}}_{\frac{p}{\ga}, \frac{q}{\ga}, \al \ga} \le C\norm{f}^{\ga}_{p, q, \al}$,  $\forall  \, f \in L^{p, q; \al}(\Om).$ 
    \item If $\tilde{p} > p,$ then $L^{\tilde{p},\tilde{q}; \be}(\Om) \hookrightarrow L^{p,q; \al}(\Om)$, i.e.,  there exists $C > 0$ such that 
    $$ \norm{f}_{p, q, \al} \le C \norm{f}_{\tilde{p},\tilde{q}, \be}, \quad \forall \, f \in L^{p,q; \al}( \Om). $$
    \item  If either $q \leq \tilde{q}$ and $\al \geq \beta$ or, $q > \tilde{q}$ and $\al + \frac{1}{q} > \beta + \frac{1}{\tilde{q}},$ then $L^{p,q; \al}( \Om) \hookrightarrow L^{p,\tilde{q};\beta}( \Om)$.
\end{enumerate}
\end{proposition}

\begin{proof}  
$(i)$ This assertion immediately follows using the definition of Lorentz-Zygmund spaces.

\noi $(ii)$ and $(iii)$ Proof follows using \cite[Theorem 9.1 and Theorem 9.3]{ColinRudnik}. 
\end{proof}

Next we give some examples of functions that lie in certain Lorentz and Lorentz-Zygmund spaces. 
\begin{example}\label{ex rearrangement}
$(i)$ For $0<d< N$, consider $g(t) = t^{-d}, t \in (0,\infty)$ and $f(x)=g(|x|), \, x \in \RN$. Since $g$ is strictly decreasing and onto, by Remark \ref{rearrangement for radial}, $f^*(t) =  g( \om_N^{-\frac{1}{N}} t^{\frac{1}{N}})= ( \frac{\om_N}{t} )^{\frac{d}{N}}.$ Consequently,
\begin{align*}
 f^{**}(t) = \frac{1}{t} \int_0^t \left( \frac{\om_N}{s} \right)^{\frac{d}{N}} \, \ds = \frac{N}{N-d} \left( \frac{\om_N}{t} \right)^{\frac{d}{N}} \,. 
 \end{align*}  Therefore, $f \in L^{\frac{N}{d}, \infty}(\R^N)$ and $\norm{f}_{\frac{N}{d}, \infty} = \displaystyle \frac{N \om_N^{\frac{d}{N}}}{N-d}$.

\noi $(ii)$ Consider $g(t)=t^{-N}( \log  ( e (\frac{R}{t} )^N ) )^{-N}, t \in (0, R_1)$ with $R_1 = R e^{\frac{1-N}{N}}$ and $f(x)=g(|x|), \, x\in B_{R_1}(0).$ One can verify that $g$ is strictly decreasing and onto. Then using Remark \ref{rearrangement for radial}, we obtain
$$f^*(t)=\frac{\om_N}{t}  \left( \log  \left( e \frac{\om_N R^N}{t} \right) \right)^{-N} \le \frac{\om_N}{t}  \left( \log  \left( e \frac{\om_N R_1^N}{t} \right) \right)^{-N}.$$
Hence
$$ \abs{f}_{1, \infty;N} = \displaystyle \sup_{0<t<|B_{R_1}(0)|} t \left( \log \left( e\frac{|B_{R_1}(0)|}{t} \right) \right)^{N} f^{*}(t) \le \om_N.$$
Therefore, $f \in L^{1, \infty;N}(B_{R_1}(0))$.
\end{example}

%% file: Hpq.tex
\section{The space $\pq$}

In this section, we provide various methods for constructing functions in $\pq$. We also discuss some properties and inclusion relations of the function space $\pq$, which will be used frequently in the subsequent sections.
\begin{proposition}\label{property1}
Let $\Om$ be an open set in $\R^N$ with $N\ge 1.$ For $t\in[0,1]$, define $g_t(x)=|g_1(x)|^t|g_2(x)|^{1-t}.$
\begin{enumerate}[(i)]
   \item For $q_1,q_2\in (0,\infty)$, let $g_1 \in \mathcal{H}_{p,q_1}(\Om)$ and $g_2\in \mathcal{H}_{p,q_2}(\Om)$.  Then $g_t\in \pq$ with $q=tq_1+(1-t)q_2.$
   \item Let $g_1 \in \pp$ and $g_2 \in L^1(\Om)$. Then for $t \in (0,1]$, $g_t \in \pq$ with $q=tp$. 
\end{enumerate}
\end{proposition}

\begin{proof}
$(i)$ For $t\in[0,1]$, let $q= tq_1 + (1-t)q_2$. For $t=0,1$, clearly, $g_t \in \mathcal{H}_{p,q}(\Om)$.  For $t \in (0,1)$,  for $u\in \c1(\Om)$, we use the H\"{o}lder's inequality to the conjugate pair $ (\frac{1}{t},\frac{1}{1-t})$ and \eqref{p-qHardy} for $g_1$ and $g_2$  to obtain the following inequalities:  
\begin{align*}
    \dis |g_t(x)||u(x)|^q & \, \dx  = \dis |g_1(x)|^t |g_2(x)|^{1-t} |u(x)|^{tq_1} |u(x)|^{(1-t)q_2} \, \dx \\ 
    & \le \left( \dis  |g_1(x)|  |u(x)|^{q_1} \, \dx \right)^t \left( \dis |g_2(x)| |u(x)|^{q_2} \, \dx  \right)^{1-t} \\
    & \le C \left(  \dis |\Gr u(x)|^p \, \dx \right)^{t\frac{q_1}{p}} \left( \dis |\Gr u(x)|^p \, \dx \right)^{(1-t) \frac{q_2}{p}}  = C \left( \dis |\Gr u(x)|^p \, \dx \right)^{\frac{q}{p}}.
\end{align*}
  
\noi $(ii)$ For $t\in(0,1]$, let $q=tp$. For $t=1$, clearly $g_t=|g_1| \in \pp$. Let $t \in (0,1)$. For $u \in \c1(\Om)$, we use the H\"{o}lder's inequality and \eqref{p-qHardy} for $g_1$ to obtain the following: 
\begin{align*}
    \dis |g_t(x)||u(x)|^q \,\dx & =  \dis |g_1(x)|^t |g_2(x)|^{1-t} |u(x)|^{tp} \,\dx \\
    & \le \left( \dis |g_1(x)|  |u(x)|^p \,\dx  \right)^t \left(  \dis |g_2(x)| \, \dx \right)^{1-t} \le C \left( \dis |\Gr u(x)|^p \, \dx \right)^{\frac{q}{p}}.
\end{align*}
\end{proof}

\begin{remark}
$(i)$ By taking $g=g_1=g_2$ in  the above proposition,  we obtain the following inclusions:
 \begin{align}
     &\mathcal{H}_{p,q_1}(\Om)\cap \mathcal{H}_{p,q_2}(\Om) \subset \pq, \quad \forall\, q\in [q_1,q_2].
      \end{align}
     In particular, 
     \begin{equation}
        \pq \supset \left\{\begin{array}{cc}
        \mathcal{H}_{p,p}(\Om) \cap L^{\infty}(\Om),   & q\in [p,p^*];\\
        \pp\cap L^1(\Om),   & q\in (0,p].
     \end{array}\right.
     \end{equation}

\noi $(ii)$ By the Sobolev inequality,  $1\in \mathcal{H}_{p,p^*}(\Om).$  Thus by Proposition \ref{property1}, for $g_1 \in  \mathcal{H}_{p,p}(\Om)$ we get  $g(x)=|g_1(x)|^{\frac{p^*-q}{p^*-p}}\in \pq$, for  $q\in[p,p^*].$
\end{remark}

\begin{proposition}\label{property2}
Let $\Om$ and $g$ be as given in \eqref{products}. For $q\in (0,p]$, let  $g_1 \in \mathcal{H}_{p,q}(\Om_1)$ and $g_2\in L^{\frac{p}{p-q}}(\Om_2)$. Then $g(x)=g_1(y)g_2(z) \in \pq.$
\end{proposition}

\begin{proof}
Let $q\in (0,p]$ and $u \in \c1(\Om)$.  Since $g_1 \in \mathcal{H}_{p,q}(\Om_1)$ and $|\Gr_{y} u(y,z)| \le |\Gr u(y,z)|$, we easily obtain
\begin{align}\label{property2-1}
    \int_{\Om}  &|g_1(y)| |g_2(z)| |u(y,z)|^q \, \dy \dz \le C \int_{\Om_2} |g_2(z)| \left( \disone |\Gr_{y}   u(y,z)|^p \, \dy \right)^{\frac{q}{p}} \, \dz \no \\ &\le  C \int_{\Om_2} |g_2(z)| \left( \disone |\Gr u(y,z)|^p \, \dy \right)^{\frac{q}{p}} \, \dz \le C  \left( \int_{\Om} |\Gr u(y,z)|^p \, \dy \dz \right)^{\frac{q}{p}} \norm{g_2}_{\frac{p}{p-q}},
    \end{align}
     where the last inequality follows from the H\"{o}lder's inequality applied to the functions $|g_2(z)|$ and $f(z):=(\int_{\Om_1} |\Gr u(y,z)|^p \, \dy )^{\frac{q}{p}}$.
     Thus $g \in \mathcal{H}_{p,q}(\Om)$.
 \end{proof}

\begin{remark} Let $q\in (0,p]$, and $\Om_2$ be bounded if $q<p$. Then $g_2=1\in L^{\frac{p}{p-q}}(\Om_2)$ for every $q\in(0,p].$  For $q\in(0,p]$ and $g_1\in \mathcal{H}_{p,q}(\Om_1)$, define $g(x)=g_1(y).$ Then  by the above proposition,  $g\in \pq.$
\end{remark}

For $\Om= \Om_1 \times \Om_2$ with $\Om \subset \Rk$ and $\Om_2 \subset \R^{N-k}$, and for $u \in \c1(\Om)$, $u_z: \Om_2 \rightarrow \R$ is the $z$-section of $u$, defined as $u_z(y) := u(y,z)$. In the following proposition, we provide a sufficient condition for $g$ to be in $\pq$ on certain symmetric domains.

\begin{proposition}\label{observ1}
Let $p \in (1,N)$ and $1\le k\le N$. For $S \subset \S^{k-1}$ and $a,b \in (0, \infty]$, let $\Om= \Om_{a,b,S} \times \RNk$, $g$ be as given in \eqref{products}.  Assume that
\begin{enumerate}[(i)]
     \item  $\g_1 \in X_1:=L^1((a,b),r^{k-1}h(r))$, for some measurable function $h:(a,b) \mapsto (0, \infty)$, and  
     \item for some $\de>0$ and $\ga \ge 0 $ with $0< \ga + \de \le 1$, 
    $$ \g_2 \in X_2:=  \left\{\begin{array}{ll}
   L^{\frac{1}{1-(\ga+ \de)}}((0, \infty),r^{N-k-1}), \quad &\text{for} \ \ \ga + \de <1; \\
    L^{\infty}((0, \infty)), \quad &\text{for} \ \ \ga + \de=1.
 \end{array} 
 \right.$$  
    \item $C=C(k,p,q)>0$ be such that for each $z \in \RNk, r \in (a,b), \om \in S$, the following inequality holds:
    \begin{align} \label{eq:obs1}
    |u_z(r\om)|^q \le C h(r) \left(\int_a^b {\tau}^{k-1} |u_z(\tau \om)|^{p^*} \, \d \tau \right)^{\ga} \left(\int_a^b {\tau}^{k-1} |\Gr_y u_z(\tau \om)|^p \, \d \tau  \right)^{\de},
\end{align}
$ \forall  \, u\in \c1(\Om)$.
\end{enumerate}
Then $g = |g_1 g_2| \in \pq$ with $q= \de p + \ga p^* $, and 
\begin{align*}
    \dis |g(x)| |u(x)|^q \, \dx \le C \norm{\g_1}_{X_1} \norm{\g_2}_{X_2} \left( \dis |\Gr u(x)|^p \, \dx  \right)^{\frac{q}{p}}, \quad \forall  \, u \in \c1(\Om),  
\end{align*}
where $C=C(N,k,p,q)>0$.
\end{proposition}

\begin{proof}
We consider two cases separately.

\noi $\bm{\underline{\ga + \de \in (0,1)}}$: Let $\ga > 0$. By noting that $(\frac{1}{\ga}, \frac{1}{\de}, \frac{1}{1-\ga - \de})$ is a conjugate triple, we integrate both sides of \eqref{eq:obs1} over $S$ and apply the generalized H\"{o}lder's inequality to get 
\begin{align}\label{i1}
   \int_{S} |u_z(r\om)|^q \, \dS \le C h(r) \left(\int_{S} \dS \right)^{1- \de- \ga} & \left(\int_{S} \int_a^b  \tau^{k-1} |u_z(\tau \om)|^{p^*} \, \d \tau \dS  \right)^{\gamma} \notag \\  & \left(\int_{S}\int_a^b \tau^{k-1} |\Gr_y u_z(\tau\om)|^p \, \d \tau \dS \right)^{\de}.  \end{align}
Multiply the above inequality by $r^{k-1} \g_1(r)$ and integrate over $(a,b)$ to get
\begin{align*}
    \int_{\Om_{a,b,S}} \g_1(|y|) |u_z(y)|^q \, \dy  \le C \left(\int_a^b \g_1(r)  r^{k-1} h(r) \, \dr \right) & \left(\int_{\Om_{a,b,S}}  |u_z(y)|^{p^*} \dy \right)^{\ga} \\
    & \left(\int_{\Om_{a,b,S}} |\Gr_y u_z(y)|^p \, \dy \right)^{\de}, 
\end{align*}
where $C=C(N,k,p,q)>0$. Now multiply both sides by $\g_2(|z|)$, integrate over $\RNk$ and then apply the H\"{o}lder's inequality, so that the above inequality and the Fubini's theorem gives
\begin{align*}
    \dis \g_1(|y|) \g_2(|z|) |u(y,z)|^q \, \dx \le C & \left( \intRnk \g_2(|z| )^{\frac{1}{1-(\ga+ \de)}} \, \dz \right)^{1-(\ga+ \de)}  \left(\int_a^b \g_1(r)  r^{k-1} h(r) \, \dr \right) \\
    & \left(\dis  |u(y,z)|^{p^*} \dx \right)^{\ga}  \left(\dis |\Gr_{y} u(y,z)|^p \, \dx \right)^{\de}. \end{align*}
Further, using the embedding of $\Dp$ into $L^{p^*}(\Om)$,  
\begin{align*}
    \dis |g(x)| |u(x)|^q \, \dx  \le C & \left( \int_0^{\infty} \g_2(r)^{\frac{1}{1-(\ga+ \de)}} r^{N-k-1} \, \dr \right)^{1-(\ga+ \de)}  \left(\int_a^b \g_1(r)  r^{k-1} h(r) \, \dr \right) \\ & \left(\dis |\Gr u(x)|^{p} \, \dx \right)^{\de + \ga \frac{p^*}{p}}, \quad \forall  \, u \in \c1(\Om). \end{align*}
Therefore, $g \in \pq$. For $\ga=0$, by noting that $(\infty, \frac{1}{\de}, \frac{1}{1-\de})$ is a conjugate triple, from  the above calculations  we  also obtain $g \in \mathcal{H}_{p,q}(\Om)$.

\noi $\bm{\underline{\ga+\de=1}}$: In this case we have $\g_2 \in L^{\infty}((0, \infty))$ and $(\frac{1}{\ga}, \frac{1}{\de}, \infty)$ is a conjugate triple.  Then proof follows along the same lines except that in \eqref{i1} (there is no term with the exponent $1-\de-\ga$).
\end{proof}

The following proposition is analogous to Proposition \ref{observ1}, where we exchange the roll of $g_1$ and $g_2$. 
\begin{proposition}\label{observ2}
Let $p \in (1,N)$ and $1\le k\le N$. For $S \subset \S^{k-1}$ and $a,b \in (0, \infty]$, let $\Om= \Om_{a,b,S} \times \mathbb{R}^{N-k}$, $g$ be as given in \eqref{products}.  Assume that
\begin{enumerate}[(i)]
     \item  for some $\de>0$ and $\ga \ge 0 $ with $0< \ga + \de \le 1$, 
          $$ \g_1 \in X_1:= \left\{\begin{array}{ll}
   L^{\frac{1}{1-(\ga+ \de)}}((a,b),r^{k-1}), \quad &\text{for} \ \ \ga + \de <1; \\
    L^{\infty}((a,b)), \quad &\text{for} \ \ \ga + \de=1, 
 \end{array} 
 \right. $$  and $\g_2 \in X_2:=L^1((0,\infty),r^{N-k-1}h(r))$, for some measurable function $h:(0,\infty) \mapsto (0, \infty)$. 
    \item $C=C(N,k,p,q)>0$ be such that for each $y \in \Om_{a,b,S}$ and $(r,\om) \in (0, \infty) \times \S^{N-k-1}$ the following inequality holds:
    \begin{align*}
    |u_y(r\om)|^q \le C h(r) \left(\int_0^{\infty} \tau^{N-k-1} |u_y(\tau\om)|^{p^*} \d \tau \right)^{\ga} 
     \left(\int_0^{\infty} \tau^{N-k-1} |\Gr_z u_y(\tau\om)|^p \, \d \tau \right)^{\de}, 
\end{align*}
$\forall  \, u\in \c1(\Om)$. 
\end{enumerate}
Then $g = |g_1 g_2| \in \pq$ with $q= \de p + \ga p^* $, and 
\begin{align*}
    \dis |g(x)| |u(x)|^q \, \dx \le C \norm{\g_1}_{X_1} \norm{\g_2}_{X_2} \left( \dis |\Gr u(x)|^p \, \dx  \right)^{\frac{q}{p}}, \quad \forall  \, u \in \c1(\Om),  
\end{align*}
where $C=C(N,k,p,q)>0$.
\end{proposition}

\begin{proof}
Proof follows from the similar set of arguments as given in the proof of Proposition \ref{observ1}.
\end{proof}

%% file: Inequality.tex
\section{The $(p,q)$-Hardy potentials ($k=N$)}
This section considers the case $k=N$ and identifies various Lorentz spaces and weighted Lebesgue spaces in $\pq$. This section contains the proof of Theorem \ref{Symmetrization} and Theorem \ref{weighted Lebesgue}. The well-definedness of $\Dp$ for the case $N \le p$ is also discussed in this section. First, we prove Theorem \ref{Symmetrization}.

\noi {\bf{Proof of Theorem \ref{Symmetrization}}:}
$(i)$ and $(ii)$: For $N \ge p$ and $q\in(0,p^*]$ (if $N>p$), $q \in (0, \infty)$ (if $N=p$), let $g, X$ be as given in Theorem \ref{Symmetrization}. Then using Proposition \ref{HSLo} and Proposition \ref{HSLo-1}, there exists $C=C(N,p,q)>0$ such that 
\begin{align*}
    \displaystyle \int_0^{|\Om|} g^*(\tau) {u^*(\tau)}^q \, \d \tau \le C \norm{g}_{X}  \left( \int_{0}^{|\Om|} {\tau}^{p - \frac{p}{N}} | {u^*}^{\prime}(\tau)|^p \, \d \tau \right)^{\frac{q}{p}}, \quad  \forall  \, u \in \c1(\Om).
\end{align*}  
Now using the Hardy-Littlewood and P\'{o}lya-Szeg\"{o} inequality (Proposition \ref{HL and PS}), we conclude that $g \in \pq$ and 
\begin{align*}
    \dis |g(x)| |u(x)|^q \,\dx \le C \norm{g}_X \left( \dis |\Gr u(x)|^p \, \dx \right)^{\frac{q}{p}}, \quad  \forall  \, u \in \c1(\Om).
\end{align*}

\noi $(iii)$ Let $N<p$, $q \in (0,\infty)$, $\Om$ be an one-sided bounded domain and $g \in L^1(\Om)$. Then
\begin{align*}
\int_{\Om} |g(x)| |u(x)|^q \ \dx \leq \norm{u}_{L^{\infty}(\Om)}^q \int_{\Om} |g(x)| \, \dx, \quad \forall  \, u \in \c1(\Om).   \end{align*}
Now using the embedding $\Dp \hookrightarrow L^{\infty}(\Om)$ (for $N<p$), $g \in \pq.$ 
\qed

\begin{remark}  $(i)$ Let $N>p$. Notice that, $\al(p,q)\le \frac{p}{p-q}, \;\text{for }\,q\in (0,p), \text{  and  } \al(p,q)\le \infty,\text{ for } q\in [p,p^*].$ Therefore, as a consequence of Theorem \ref{Symmetrization} and the inclusions of the Lorentz spaces ($(iii)$ of Proposition \ref{Lorentz properties}) we obtain the following Lebesgue spaces in $\pq$: $$L^{\al(p,q)}(\Om)\subset \pq, \quad \text{ for }  q\in (0,p^*].$$

\noi $(ii)$ Using the inclusions of Lorentz spaces, we have 
     $$\pq \supset \left\{ \begin{array}{ll}
          L^{\al(p,q),s}(\Om), & \quad  \text{ for }  q\in (0,p), s \in \left[1,\frac{p}{p-q}\right]; \\
             L^{\al(p,q), s}(\Om), & \quad  \text{ for }  q \in [p,p^*], s\in [1,\infty]. \\
              \end{array}\right.$$
              
\noi $(iii)$ If $\Om$ is a bounded domain and $q\in (0,p^*]$, then we have $$L^{r,s}(\Om)\subset \pq, \quad r \in [\al(p,q),\infty], s\in [1,\infty].$$
In particular, $L^r(\Om)\subset \pq$ for $r\in [\al(p,q),\infty].$
\end{remark}

Now we identify various weighted Lebesgue spaces in $\pq$.
Recall that, for $S \subset \S^{N-1}$ and $a,b\in [0,\infty]$ with $a<b$, we set $\Om_{a,b,S}= int\left(\{x\in \RN: a \leq |x| < b, \frac{x}{|x|}\in S \ \text{if} \ x \neq 0 \}\right).$  

\noi \textbf{Proof of Theorem \ref{weighted Lebesgue}:} $(i)$ Let $\Om= \Om_{a,b,S}$ and $u \in \c1(\Om)$. For $\tau \in (a,b)$ and $\om \in S$, using the polar decomposition we define $\vph(\tau)= u(\tau \omega)$. We consider three separate cases.

\noi $\bm{ \underline{ q \in (0, p)}}$: Using the fundamental theorem of integration, we write 
\begin{align*}
\vph(r) = - \displaystyle \int_r^b \vph'(\tau) \, \d \tau = - \int_r^b \vph'(\tau)  \tau ^{\frac{N-1}{p}} \tau^{\frac{1-N}{p}} \, \d \tau.
\end{align*}
By the H\"{o}lder's inequality, 
\begin{align*}
    |\vph(r)| \le \left(\int_r^b \tau^{\frac{1-N}{p-1}} \, \d \tau  \right)^{\frac{1}{\p}} \left(\int_r^b \tau^{N-1} |\vph'(\tau)|^p \, \d \tau \right)^{\frac{1}{p}}.
\end{align*}
The above inequality yields
\begin{align*}
|\vph(r)|^q \leq \left( \frac{p-1}{N-p} \right)^{\frac{q}{\p}} r^{\frac{(p-N)q}{p}} \left( \int_r^b \tau^{N-1} | \vph'(\tau) |^p \, \d \tau \right)^{\frac{q}{p}}.
\end{align*}
Set $C = \left( \frac{p-1}{N-p} \right)^{\frac{q}{\p}}$. Now $\vpp(\tau) = \Gr u(\tau\omega) \cdot \omega.$ Hence for each $\om \in S$, 
\begin{align}\label{radial 1}
|u(r\om)|^q \leq C r^{\frac{(p-N)q}{p}} \left( \int_a^b \tau^{N-1} | \Gr u(\tau \om)|^p \, \d \tau \right)^{\frac{q}{p}}, \quad \forall \, u \in \c1(\Om).
\end{align}
Thus for $q \in (0, p)$ and $\g \in L^1((a, b), r^{\frac{N}{\al(p,q)} - 1})$, by taking $\ga = 0$, $\de = \frac{q}{p}$, and $h(r)=r^{\frac{(p-N)q}{p}}$ one can verify that all the assumptions of Proposition \ref{observ1} are satisfied. Therefore, $g \in \pq$ for $q \in (0,p)$.

\noi $\bm{ \underline{ q \in  \left[ p,P^*(1) \right)}}$: In this case, for $q \ge p$ using the fundamental theorem of integration,
\begin{align*}
|\varphi(r)|^q =  - \int_r^b \frac{\d}{\d \tau} |\vph(\tau)|^q \, \d \tau  & = - q \displaystyle \int_r^b |\varphi(\tau)|^{q-1} \vpp(\tau) \, \d \tau \\
& = - q \displaystyle \int_r^b |\varphi(\tau)|^{q -1} \vpp(\tau) \tau^{\frac{N-1}{p}}\tau^{\frac{1-N}{p}} \,  \d \tau.
 \end{align*}  
Using the generalized H\"{o}lder's inequality we estimate the right hand side of the above identity as 
\begin{align}\label{radial 2}
|\varphi(r)|^q \le q \left(  \int_r^b  \tau^{dp_1} \, \d \tau \right)^{\frac{1}{p_1}}   \left(  \int_r^b  \tau^{N-1} |\varphi(\tau)|^{(q-1)p_2} \,  \d \tau \right)^{\frac{1}{p_2}} \left( \int_r^b \tau^{N-1} \abs{\vpp(\tau)}^p \, \d \tau \right)^{\frac{1}{p}},
\end{align}
where $(p,p_1,p_2)$ is a conjugate triple and $d = (1-N)( \frac{1}{p} + \frac{1}{p_2})$. We set $p_2 = \frac{p^*}{q-1}$ and hence $p_1 = \frac{Np}{N(p-q) + qp -p}.$ Now the first integral of \eqref{radial 2} can be estimated as
\begin{align*}
   \int_r^b  \tau^{d p_1} \, \d \tau  \le  -\frac{r^{dp_1+1}}{dp_1+1},
\end{align*}
where  $dp_1 + 1 = \frac{qp_1(p-N)}{p}<0$. Set $C=-\frac{1}{dp_1 + 1}$. Thus for each $\om \in S$, \eqref{radial 2} yields 
\begin{align}\label{radial 3}
    |u(r \om)|^q \leq C  q r^{\frac{q(p-N)}{p}}  \left(  \int_a^b  \tau^{N-1} |u(\tau \om)|^{p^*} \,  \d \tau \right)^{\frac{1}{p_2}} \left( \int_a^b \tau^{N-1} \abs{\Gr u(\tau \om)}^p\,  \d \tau \right)^{\frac{1}{p}},
\end{align}
$ \forall \, u \in \c1(\Om)$. For $q \in [p, P^*(1))$ and $\g \in L^1((a, b), r^{\frac{N}{\al(p,q)} - 1})$, by taking  $\ga = \frac{1}{p_2}$, $\de = \frac{1}{p}$, and $h(r)=r^{\frac{q(p-N)}{p}}$, one can verify that all the assumptions of Proposition \ref{observ1} are satisfied. Therefore, $g \in \pq$ for $q \in [ p,P^*(1))$.

\noi $\bm{\underline{ q \in  \left[P^*(1), p^* \right]}}$: Let $\g$ be strictly decreasing. Then for $u \in \cc(\Om)$, using the Hardy-Littlewood inequality (Proposition \ref{HL and PS}) and Remark \ref{rearrangement for radial}, we get
\begin{equation} \label{radial 4}
    \dis \g(|x|) |u(x)|^q \, \dx \leq \int_0^{|\Om|} \g (\om_N^{-\frac{1}{N}} r^{\frac{1}{N}}) u^{*}(r)^q \, \dr. 
\end{equation}
For $u \in \cc(\Om)$, by P\'{o}lya-Szeg\"{o} inequality (Proposition \ref{HL and PS}) $u^* \in W^{1,p}((0,|\Om|))$ and hence $u^*$ is absolutely continuous. Thus using the fundamental theorem of integration, we write
\begin{align}
\no {u^*(r)}^{q} = - q \displaystyle \int_{r}^{\infty} {u^*(\tau)}^{q -1} {u^*}'(\tau) \, \d \tau = - q \displaystyle \int_{r}^{\infty} {u^*(\tau)}^{q -1} {u^*}'(\tau) \tau^{\frac{N-1}{N}}\tau^{\frac{1-N}{N}} \,  \d \tau.
 \end{align}
 Therefore, by the  H\"{o}lder's inequality, 
\begin{align}\label{radial 5}
{u^*(r)}^{q} \le q \left(  \int_{r}^{\infty}  \tau^{\frac{1-N}{N} \p} {u^*(\tau)}^{(q-1)\p} \, \d \tau \right)^{\frac{1}{\p}}   \left( \int_{r}^{\infty}  \tau^{\frac{N-1}{N}p} \abs{{u^*}'(\tau)}^p\,  \d \tau \right)^{\frac{1}{p}}.
\end{align}
 Notice that, for $q \ge P^*(1)$, $\frac{(q-1)\p}{p^*} \ge 1$ and also $\frac{1-N}{N} - \frac{q-1}{p^*} +\frac{1}{\p} = \frac{1}{\al(p,q)}-1.$  Now we estimate the first integral of \eqref{radial 5} using Proposition \ref{Mazya1} as shown below:
\begin{align*}
\no & \frac{p^*}{(q-1)\p}  \int_r^{\infty} \tau^{\frac{1-N}{N} \p }\left( {u^*(\tau)}^{p^*} \right)^{\frac{(q-1)\p}{p^*}} \, \d \tau = \int_r^{\infty} \tau^{\frac{1-N}{N} \p-{\frac{(q-1)\p}{p^*}+1}}  ({u^*(\tau)}^{p^*})^{\frac{(q-1)\p}{p^*}} \,{\rm d}\left(\tau^{\frac{(q-1)\p}{p^*}}\right) \\
    & \le  r^{\frac{1-N}{N}\p - \frac{(q-1)p'}{p^*} +1}
\int_r^{\infty} ({u^*(\tau)}^{p^*})^{\frac{(q-1)\p}{p^*}} \, {\rm d}\tau^{\frac{(q-1)\p}{p^*}}\le r^{(\frac{1}{\al(p,q)}-1)p'}  \left( \int_0^{\infty} {u^*(\tau)}^{p^*}\d \tau \right)^{\frac{(q-1)p'}{p^*}}. 
\end{align*}
Hence using \eqref{radial 5} we obtain
\begin{align*}
  {u^*(r)}^{q} \le C  r^{\frac{1}{\al(p,q)}-1} \left( \int_0^{\infty} {u^*(\tau)}^{p^*} \, \d \tau \right)^{\frac{q-1}{p^*}}  \left( \int_{r}^{\infty}  \tau^{\frac{N-1}{N}p} \abs{{u^*}'(\tau)}^p\,  \d \tau \right)^{\frac{1}{p}},
\end{align*}
where $C=C(N,p,q)>0$. Next multiply the above inequality by $\g (\om_N^{-\frac{1}{N}} r^{\frac{1}{N}})$ and integrate over $(0,\infty)$ to get
\begin{align*}
    \int_0^{\infty} \g (\om_N^{-\frac{1}{N}} r^{\frac{1}{N}}) {u^*(r)}^{q} \, \dr \le \frac{C}{N\om_N^{\frac{1}{N}}}& \left( \int_0^{\infty} \g (\om_N^{-\frac{1}{N}} r^{\frac{1}{N}}) r^{\frac{1}{\al(p,q)}-1} \, \dr \right) \\   & \times \left( \int_0^{\infty} {u^*(\tau)}^{p^*} \, \d \tau \right)^{\frac{q-1}{p^*}}  \left( \int_{r}^{\infty}  \tau^{\frac{N-1}{N}p} \abs{{u^*}'(\tau)}^p\,  \d \tau \right)^{\frac{1}{p}}.
\end{align*}
By change of variable, 
\begin{align*}
    \int_{0}^{\infty}  \g(\om_N^{-\frac{1}{N}} r^{\frac{1}{N}})     r^{\frac{1}{\al(p,q)}-1}  \, \dr = \om_N^{\frac{1}{N}+\frac{1}{\al(p,q)}-1} \int_{0}^{\infty} \tau^{\frac{N}{\al(p,q)}-1} \g(\tau) \, \d \tau.
\end{align*}
Therefore, using the P\'{o}lya-Szeg\"{o} inequality (Proposition \ref{HL and PS}) and \eqref{radial 4}, we obtain that
\begin{align*}
    \dis \g(|x|) |u(x)|^q \, \dx \le C \left(  \int_{0}^{\infty} r^{\frac{N}{\al(p,q)}-1} \g(r) \, \dr \right) &\left( \dis |u(x)|^{p^*} \, \dx \right)^{\frac{q-1}{p^*}} 
    \left( \intRn |\Gr u(x)|^p \, \dx \right)^{\frac{1}{p}},
\end{align*}
for some $C= C(N,p,q)>0$. Further, the embedding $\dpp(\Om) \hookrightarrow L^{p^*}(\Om)$ gives
\begin{align*}
    \dis \g(|x|) |u(x)|^q \, \dx \le C \left(  \int_{0}^{\infty} r^{\frac{N}{\al(p,q)}-1} \g(r) \, \dr \right) \left( \dis |\Gr u(x)|^p \, \dx \right)^{\frac{q}{p}}, \quad \forall \, u \in \cc(\Om).
\end{align*}
Therefore, $g \in \pq$ for $q \in [P^*(1),p^*]$.

\noi $(ii)$ and $(iii)$: Let $\Om=\Om_{a,b,S}$ with $a>0$ if $N=p$ and $a \ge 0$ if $N<p$. For $\tau \in (a,b)$ and $\om \in S$, we define $\vph(\tau)= u(\tau \omega)$. Then we express
\begin{align*}
\vph(r) = \displaystyle \int_a^r \vph'(\tau) \, \d \tau =  \int_a^r \vph'(\tau)  \tau^{\frac{N-1}{p}} \tau^{\frac{1-N}{p}} \, \d \tau \,.
\end{align*}
Thus,
\begin{align*}
    |\vph(r)| \le \left(\int_a^r \tau^{\frac{1-N}{p-1}} \, \d \tau  \right)^{\frac{1}{\p}} \left(\int_a^r \tau^{N-1} |\vph'(\tau)|^p \, \d \tau \right)^{\frac{1}{p}} \,,
\end{align*}
and for each $\om \in S,$ 
\begin{align*}
|u(r\om)|^q \le \left\{\begin{array}{ll}
    \displaystyle \left( \frac{p-1}{p-N} \right)^{\frac{q}{\p}} r^{\frac{(p-N)q}{p}} \left( \int_a^b \tau^{N-1} | \Gr u(\tau\om) |^p \, \d \tau \right)^{\frac{q}{p}}, \quad &\text{if} \ \ N<p;  \\
    \displaystyle \left( \log \left( \frac{r}{a} \right) \right)^{\frac{q}{\p}} \left( \int_a^b \tau^{N-1} |\Gr u(\tau\om)|^p \, \d \tau \right)^{\frac{q}{p}}, \quad &\text{if} \ \ N=p,
 \end{array} 
 \right. 
\end{align*}
$\forall \, u \in \c1(\Om)$. Clearly, by taking $h(r)=r^{\frac{(p-N)q}{p}}$ (if $N<p$), $\left( \log \left( \frac{r}{a} \right) \right)^{\frac{q}{\p}}$ (if $N=p$),  $\ga = 0$, and $\de=\frac{q}{p}$, all the assumptions of Proposition \ref{observ1} are satisfied. Therefore, $g \in \pq$ for $q\in (0,p]$. \qed
\begin{remark} $(i)$ Suppose  $\Om_0 \subset B_a[x_0]^c$  for some $x_0 \in \R^N$ where $a>0$ for $N=p$, and $a=0$ for $N<p$. Let $g_0\in L^1_{loc}(\Om_0)$. Take $\Om=-x_0+\Om_0$ and $g(x)=g_0(x+x_0)$ for $x\in \Om.$ Now, if the zero extension of  $g$ to $ B_a[0]^c$ satisfies  one of the assumptions  of the above theorem, then it is easy to see that $g_0\in \mathcal{H}_{p,q}(\Om_0).$

\noi $(ii)$ For a measurable function $h: \Om_{a,b,S} \mapsto \R$ with $\h(r) \leq \g(r)$, $\forall \, r \in (a,b)$, where $\g:(a,b) \mapsto \R$ as in Theorem \ref{weighted Lebesgue}, we  get $h \in \mathcal{H}_{p,q}(\Om_{a,b,S})$. In this way, for $q \in [P^*(1), p^* ]$, we can relax the strictly monotonicity (decreasing) of $\h$.
\end{remark}
 
\noi \textbf{Proof of Corollary \ref{function space}}: Let $N \le p$ and $\Om=\RN \setminus B_a[0]$. We choose 
\begin{align*}
    w(r) := \left\{ \begin{array}{ll}
         \left(r^{N+1} \log(\frac{r}{a})^{(N-1)}\right)^{-1}, & \quad  \text{ for }  N=p, a>0; \vspace{0.2 cm} \\
             (1+r)^{-(p+1)}, & \quad  \text{ for }  N<p, a=0. \\
              \end{array}\right.
\end{align*}
It is easy to see that $w\in L^1((a, \infty), (r\log ( \frac{r}{a} ) )^{N-1})$ (if $N=p$) and $w \in L^1((0,\infty),r^{p - 1})$ (if $N<p$). Therefore, by Theorem \ref{weighted Lebesgue}, 
\begin{align*}
   \int_{K} |u(x)|^p \, \dx \le \left(\frac{1}{\displaystyle \inf_{x \in K} w(|x|)} \right)\dis w(|x|) |u(x)|^p \, \dx \le C \dis |\Gr u(x)|^p \, \dx, \quad \forall \, u \in \Dp,  
\end{align*}
where $K \subset \Om$ is compact and $C=C(N,p,K)>0$. Further, $|\Gr u| \in L^p(\Om)$ for $u \in \Dp$. Hence from the above inequality we get $\Dp \subset W_{loc}^{1,p}(\Om)$, and 
\begin{align*}
     \int_{K} \left( |u(x)|^p + |\Gr u(x)|^p \right) \, \dx \le C \dis |\Gr u(x)|^p \, \dx, \quad \forall \, u \in \Dp,
\end{align*}
where $C=C(N,p,K)>0$.  \qed

\section{The cylindrical $(p,q)$-Hardy potentials}
In this section, we identify product of functions from certain Lorentz spaces and weighted Lebesgue spaces in $\mathcal{H}_{p,q}(\Om_1 \times \Om_2)$, where $\Om_1 \subset \Rk$ and $\Om_2 \subset \RNk$. This section contains the proof of Theorem \ref{cylin C-K-N}, Theorem \ref{Lorentz product}, and  Theorem \ref{weighted Lebesgue product}.

\noi \textbf{Proof of Theorem \ref{cylin C-K-N}:} \label{Proof 1.10} 
For $S \subset \S^{k-1}$, let $\Om_1=\Om_{a,b,S}$, and $\Om, g$ be as in \eqref{products}.  For $u \in \c1(\Om)$, as before, we let $u_z(y)= u(y,z), \, \forall  \, y \in \Om_{a,b,S}$. For a fixed $\om\in S,$  define $\vph(r) = u_z(r \om)$ for $r \in (a,b)$ and $\vph(r)=0$ for $0\le r\le a$ and for $r\ge b$. 

\noi $(i)$ Our proof follows the similar arguments as in the proof of \cite[Theorem 2.1]{BT}.  By fundamental theorem of integration,  for  $q>1$ and $r\in (a,b),$ we can write 
\begin{align*}
|\vph(r)|^q = -\int_1^{\infty} \frac{\d}{\dl} |\vph(\la r)|^q \, \dl & = -q \int_1^{\infty} |\vph(\la r)|^{q-1} \vpp(\la r) r \, \dl.
\end{align*}
Now $\vph'(r)=\Gr_y u_z(r\om) \cdot \om$. Thus for each $\om \in S$ and $z \in \RNk$, 
\begin{align*}
    |u_z(r \om)|^q = - q \int_1^{\infty} |u_z(\la r \om)|^{q-1} \Gr_y u_z(\la r \om) \cdot \om r \, \dl.
\end{align*}
Next we multiply both sides of the above inequality  by $r^{k-1-s}$ and integrate over $\Om_2 \times S \times (a,b)$ to obtain
\begin{align*}
  \dis &  \frac{|u(x)|^q}{|y|^s} \, \dx \le q \int_1^{\infty} \dl \int_{\Om_2} \dz \int_{S} \dS \int_a^b r^{k-1-s} |u(\la r \om, z)|^{q-1}  |\Gr_{y} u(\la r \om, z)| r  \, \dr \no \\
  & = q \int_1^{\infty} \dl \int_{\Om_2} \dz \int_{S} \dS  \int_{\la a}^{\la b} \left( \frac{\rho}{\la} \right)^{k-s} |u(\rho \om, z)|^{q-1}  |\Gr_{y} u(\rho \om, z)| \, \frac{\d \rho}{\la} \no \\
  & \le q \int_1^{\infty} \frac{\dl}{\la^{k+1-s}} \int_{\Om_2} \dz \int_{S} \dS \int_{ a}^b {\rho}^{k-s}  |u(\rho \om, z)|^{q-1}  |\Gr_{y} u(\rho \om, z)| \, \d \rho, 
\end{align*}  
where the last inequality holds since $\la >1$ and  $u(r \om,z)=0$ for $ r\ge b$. Observe that, \begin{equation} \label{qineq}
\displaystyle \int_1^{\infty} \frac{\dl}{\la^{k+1-s}}<\infty \Longleftrightarrow    s<k  \Longleftrightarrow  \frac{N(p-q)+qp}{p}<k \Longleftrightarrow q> \frac{p(N-k)}{N-p}=P^*(k).
\end{equation}  Thus for $q>P^*(k)$ and $s=\frac{N}{\al(p,q)},$ we get 
\begin{equation}\label{eq:H1}
    \dis \frac{|u(x)|^q}{|y|^s}   \, \dx  \le C \dis  \frac{1}{|y|^{s-1}}  |u(y, z)|^{q-1}  |\Gr u(y, z)|  \, \dx.
\end{equation}
Next we estimate  the right hand side of \eqref{eq:H1} for  $ q\in [p, P^*(1)]$ i.e., for $s \in [1,p]$. If $s\in (1,p)$, then one can verify that  $(\frac{s}{s-1},\frac{sp}{p-s}, p)$ is a conjugate triple and  $$\left(\frac{q}{s}-1 \right)\frac{sp}{p-s} = p\frac{q-\frac{N}{\al(p,q)}}{p-\frac{N}{\al(p,q)}} = p\frac{Nq-Np}{p^2 - Np+Nq-pq} = \frac{Np}{N-p}.$$ Now using the H\"older's inequality we get
 \begin{align}
 \dis  \frac{1}{|y|^{s-1}}  |u(y, z)|^{q-1}  |\Gr u(y, z)| & \, \dx  
 = \dis  \frac{1}{|y|^{s-1}}  |u(x)|^{q(\frac{s-1}{s})+\frac{q}{s}-1}  |\Gr u(x)|  \, \dx \no \\
& \le  \left( \dis \frac{|u(x)|^q}{|y|^s}  \, \dx \right)^{\frac{s-1}{s}} \left( \dis |u(x)|^{p^*} \, \dx \right)^{\frac{p-s}{sp}} \left( \dis \abs{\Gr u(x)}^p \, \dx \right)^{\frac{1}{p}}\no \\
& \le C \left( \dis \frac{|u(x)|^q}{|y|^s}  \, \dx \right)^{\frac{s-1}{s}} \left( \dis \abs{\Gr u(x)}^p \, \dx \right)^{\frac{q}{ps}},\label{eq:Tarantello1}
\end{align}
where the last inequality follows since $(\frac{p-s}{sp}) \frac{p^*}{p} + \frac{1}{p} = \frac{N-s}{(N-p)s}=\frac{q}{ps}.$  If $s=1$ or $s=p$, then the above conjugate triple become $(\infty,p',p)$ or $(p',\infty,p)$ and the similar estimates as above yields \eqref{eq:Tarantello1} for $s=1,p.$ Further, if $s<k$ (equivalently, if  $q>P^*(k)$), then \eqref{eq:H1} yields 
\begin{align}\label{eq:H2}
    \dis \frac{|u(x)|^q}{|y|^{\frac{N}{\al(p,q)}}}   \, \dx  \le C \left( \dis \abs{\Gr u(x)}^p \, \dx \right)^{\frac{q}{p}}, \quad \forall\, u \in \c1(\Om).
\end{align}
Therefore, 
\begin{equation}\label{eq:1}
    |y|^{-\frac{N}{\al(p,q)}} \in \pq, \quad q\in \left[p, P^*(1)\right] \text{ with } q> P^*(k).
\end{equation}
\noi Next we consider $q \in (P^*(1),p^*).$ In this case one can write $q= tP^*(1)+ (1-t)p^*$ for some $t\in (0,1).$  Therefore,
$$t=\frac{q-p^*}{P^*(1)-p^*}=\frac{(N-p)q-Np}{-p}=\frac{N}{\al(p,q)}=s.$$
Now we apply the H\"older's inequality and use \eqref{eq:H2} for $q=P^*(1)$ to get
\begin{align*}
    \dis  \frac{|u(x)|^q}{|y|^s} \dx & \le \left( \dis \frac{|u(x)|^{P^*(1)}}{|y|} \, \dx  \right)^{s} \left(  \dis |u(x)|^{p^*} \, \dx \right)^{1-s} \le C \left(  \dis |\Gr u(x)|^p \, \dx \right)^{\frac{q}{p}},
\end{align*}
$ \forall \, u \in \c1(\Om)$. For $q=p^*,$ the above inequality follows from the Sobolev inequality. Therefore, 
\begin{equation}\label{eq:2}
    |y|^{-\frac{N}{\al(p,q)}} \in \pq, \quad q\in \left (P^*(1),p^* \right].
\end{equation}
Now by \eqref{eq:1}  and \eqref{eq:2}, we conclude that 
$$    |y|^{-\frac{N}{\al(p,q)}} \in \pq,  \text{ for} \;
    q \in \left\{ \begin{array}{ll}
    [p,p^*],   & \  k >p;  \\
    (P^*(k),p^*],  & \ k \leq p. 
\end{array}\right.
$$

\noi $(ii)$ Let $\Om = \Om_{a,b,S} \times \Om_2$ where $0 \not \in \Om_{a,b,S}$. Let $u \in \c1(\Om)$ and $\vph, s$ be as in $(i).$ Then for $q>1$ we write 
\begin{align*}
   |\vph(r)|^q = \int_0^1 \frac{\d}{\dl} |\vph(\la r)|^q \, \dl = q \int_0^1 |\vph(\la r)|^{q-1} \vpp(\la r) r \, \dl.
\end{align*}
Thus for each $\om \in S$ and $z \in \RNk$,
\begin{align*}
 |u_z(y)|^q = q \int_0^1 |u_z(\la r \om)|^{q-1} \Gr_y u_z(\la r \om) \cdot \om r \, \dl.
\end{align*}
Multiply the above inequality $r^{k-1-s}$ and integrate over $\Om$ to get
\begin{align*}
   \dis & \frac{|u(x)|^q}{|y|^s} \, \dx \le  q \int_0^1 \dl \int_{\Om_2}  \dz \int_S \dS \int_a^b  r^{k-1-s} |u( \la r \om, z)|^{q-1}  |\Gr_y u( \la r \om, z)| r \, \dr \\
   & = q \int_0^1 \dl \int_{\Om_2} \dz \int_S \dS \int_{\la a}^{\la b}  \left( \frac{\rho}{\la} \right)^{k-1-s} |u(\rho \om,z)|^{q-1}  |\Gr_y u(\rho \om,z)| \frac{\rho}{\la} \, \frac{\d \rho}{\la}  \\
   & \le q \left( \int_0^1 \frac{\dl}{\la^{k+1-s}} \right) \int_{\Om_2} \dz \int_S \dS \int_a^{b}  {\rho}^{k-1-s} |u(\rho \om,z)|^{q-1}  |\Gr_y u(\rho \om,z)|  \rho \, \d \rho,
\end{align*}
where the last inequality holds since $\la <1$ and $u(r \om,z)=0$ for $r\le a$. Notice that
\begin{equation*}
  \int_0^1 \frac{\dl}{\la^{k+1-s}}< \infty \Longleftrightarrow s>k \Longleftrightarrow q < P^*(k).  
\end{equation*} 
Thus for $q<P^*(k)$ the above inequality yields 
\begin{equation*}
    \dis \frac{|u(x)|^q}{|y|^s} \, \dx  \le C \dis  \frac{1}{|y|^{s-1}}  |u(x)|^{q-1}  |\Gr u(x)| \, \dx.
\end{equation*}
Now, for $q \in [p, P^*(k))$ we estimate the right hand side of the above inequality as before. Clearly, in this range of $q$ we have $s \in (k,p]$. Moreover, since $k \geq 1$, we also have $s \in (1,p]$ and hence following the arguments that gives \eqref{eq:H2}, we obtain 
\begin{align}\label{eq:H4}
   \dis \frac{|u(x)|^q}{|y|^\frac{N}{\al(p,q)}} \, \dx  \le C \left( \dis \abs{\Gr u(x)}^p \, \dx \right)^{\frac{q}{p}}, \quad \forall  \, u \in \c1(\Om).
\end{align}
Therefore, $|y|^{-\frac{N}{\al(p,q)}} \in \pq$ for $q \in [p, P^*(k))$. Next we consider $q = P^*(k)$. In this case, we have $\frac{N}{\al(p,q)}=k$  and $\frac{p(q-k)}{p-k}=p^*$. For $k<p$, and $u \in \c1(\Om)$, we apply the H\"older's inequality with the conjugate pair $(\frac{p}{k},\frac{p}{p-k})$ to get
\begin{align*}
\int_{\Om} \frac{|u(x)|^q}{|y|^{\frac{N}{\al(p,q)}}} \ \dx & = \int_{\Om} \frac{|u(x)|^k}{|y|^k} |u(x)|^{q-k} \ \dx \leq \left( \int_{\Om} \frac{|u(x)|^p}{|y|^p} \ \dx\right)^{\frac{k}{p}}  \left( \int_{\Om} |u(x)|^{p^*} \ \dx\right)^{\frac{p-k}{p}}.
\end{align*}
Further, using \eqref{eq:H4}, $|y|^{-p} \in \pp$. Therefore, using the embedding $\Dp \hookrightarrow L^{p^*}(\Om)$, we get the right hand side of the above inequality is lesser than $\left(\int_{\Om} |\nabla u|^p \right)^{\frac{q}{p}}$. Thus, $|y|^{-\frac{N}{\al(p,q)}} \in \pq$ for $q=P^*(k)$. \qed

\noi \textbf{Proof of Corollary \ref{gen cylin C-K-N}:} Let $g(x) = g_1(y)$ and $\g_1 \in L^{\infty}( (a,b),r^{\frac{N}{\al(p,q)}}).$ If $|y|^{-\frac{N}{\al(p,q)}} \in \pq$, then
\begin{align*}
\dis |g(x)||u(x)|^q \, \dx & \le \esssup_{y \in \Om_1} \g_1(|y|) |y|^{\frac{N}{\al(p,q)}}  \dis \frac{|u(x)|^q}{|y|^{\frac{N}{\al(p,q)}}}  \, \dx  \\ & \le C\norm{\g_1}_{L^{\infty}( (a,b),r^{\frac{N}{\al(p,q)}})} \left( \dis \abs{\Gr u(x)}^p \, \dx \right)^{\frac{q}{p}}, \quad \forall  \, u \in \c1(\Om),
\end{align*}
for some $C=C(N,k,p,q)>0$, i.e.,  $g \in \pq$. Thus the same conclusions of Theorem \ref{cylin C-K-N} hold for $g$ in place of $|y|^{-\frac{N}{\al(p,q)}}.$

Next we proceed to prove Theorem \ref{Lorentz product}. For proving Theorem \ref{Lorentz product} we require the following proposition: 
\begin{proposition}\label{cylin Lorentz prop}
Let $p \in (1,N)$. Let $\Om=\Om_1\times \Om_2$ where $\Om_1$ and $\Om_2$ be any two open sets in $\Rk$ and $\RNk$ respectively. Let $\Om_1$ be bounded if $k \leq p$, and $\Om$ be bounded for $q\in (0,p)$. Let $t\in[0,1]$ and 
\begin{align*}
    g_1 \in X := \left\{\begin{array}{ll}
    L^{\frac{k}{tp},\infty}(\Om_1), \quad &\text{if} \ \ k>p; \\
    L^{\frac{1}{t},\infty,t(k-1)} (\Om_1), \quad &\text{if} \ \ k=p; \\
    L^{\frac{1}{t}}(\Om_1), \quad &\text{if} \ \ k<p.
 \end{array} 
 \right. 
\end{align*}
Then for $q=tp+(1-t)p^*$ or  $q=tp$ with $t>0,$  $g(x)=g_1(y) \in \pq.$  Moreover, 
\begin{align*}
    \dis |g_1(y)| |u(x)|^q \, \dx \le C \norm{g_1}_X \left(  \dis |\Gr u(x)|^p \, \dx \right)^{\frac{q}{p}}, \quad \forall  \, u \in \c1(\Om),
\end{align*}
for some $C=C(N,k,p,q)>0$.
\end{proposition}

\begin{proof} 
Let $t \in [0,1]$ and $q=tp+(1-t)p^*$. For $t=0$, we have $q=p^*$ and $g_1 \in L^{\infty}(\Om_1)$. Hence the proof follows from the Sobolev embedding. For $t\in (0,1]$, and $u \in \c1(\Om)$, we apply the H\"{o}lder's inequality to get
\begin{align}\label{cylinnn 3}
  \int_{\Om} |g_1(y)||u(x)|^q   \, \dx \le  \left( \int_{\Om} |g_1(y)|^{\frac{1}{t}} |u(x)|^p \, \dx  \right)^{t} \left( \int_{\Om} |u(x)|^{p^*} \, \dx \right)^{1-t}.
\end{align}
Using $(i)$ of Proposition \ref{Lorentz properties} and Proposition \ref{LZ prop}, \begin{align*}    |g_1|^{\frac{1}{t}} \in Y= \left\{\begin{array}{ll}
    L^{\frac{k}{p},\infty}(\Om_1), \quad &\text{if} \ \ k>p; \\
    L^{1,\infty,k-1} (\Om_1), \quad &\text{if} \ \ k=p; \\
    L^{1}(\Om_1), \quad &\text{if} \ \ k<p,
 \end{array} \right. 
\end{align*}
and $\norm{|g_1|^{\frac{1}{t}}}_{Y} \le C \norm{g_1}_{X}^{\frac{1}{t}}$ for some $C>0$. Consequently, by Theorem \ref{Symmetrization}, it follows that 
$|g_1|^{\frac{1}{t}} \in \mathcal{H}_{p,p}(\Om_1)$ and
\begin{align*}
    \int_{\Om_2} \int_{\Om_1} |g_1(y)|^{\frac{1}{t}} |u(y,z)|^p \, \dy \dz & \le C \norm{|g_1|^{\frac{1}{t}}}_{Y} \int_{\Om_2} \int_{\Om_1} |\Gr_{y} u(y,z)|^p \, \dy \dz \\
    & \le C \norm{g_1}_{X}^{\frac{1}{t}} \int_{\Om_2} \int_{\Om_1} |\Gr u(y,z)|^p \, \dy \dz,
\end{align*}
where $C=C(k,p,q)>0$. Therefore, using the embeddings of $\Dp$ into $L^{p^*}(\Om)$ in \eqref{cylinnn 3}, we get
\begin{align*}
  \dis |g_1(y)| |u(x)|^q \, \dx & \le C  \norm{g_1}_{X}\left( \dis |\Gr u(x)|^p \, \dx \right)^{t + \frac{(1-t)p^*}{p}} 
   = C \norm{g_1}_{X} \left( \dis |\Gr u(x)|^p \, \dx \right)^{\frac{q}{p}},
\end{align*}
$\forall \, u \in \c1(\Om)$ and for some $C=C(N,k,p,q)>0$. For $q=tp$ with $t\in(0,1]$, using the boundedness of $\Om$, the result follows along the same line except that in \eqref{cylinnn 3} (there is no term with the exponent $1-t$).
\end{proof}

\begin{remark}
Let $t \in [0,1]$ be such that $q = tp+(1-t)p^*$. Then 
$$ \frac{k}{tp}= \frac{k}{p} \left( \frac{p^*}{p^*-q} - \frac{p}{p^*-q}  \right)= k \al(p,q) \left( \frac{1}{p} - \frac{1}{p^*} \right)= \frac{\al(p,q)k}{N}.$$
Further, from $(i)$ of Example \ref{ex rearrangement}, $g_1(y)=\abs{y}^{-\frac{N}{\al(p,q)}} \in L^{\frac{\al(p,q)k}{N},\infty}(\Rk)$, for $k > \frac{N}{\al(p,q)}$. Moreover, it is easy to see that for $q \ge p$, $p \ge \frac{N}{\al(p,q)}$. Thus for $k>p$ and $q \in [p,p^*]$, from the above proposition we get the following generalized cylindrical C-K-N inequality: 
\begin{align*}
    \intRn |g_1(y)| |u(x)|^q \, \dx \le C \norm{g_1}_{\frac{\al(p,q)k}{N}, \infty} \left( \intRn |\Gr u(x)|^p \, \dx \right)^{\frac{q}{p}}, \quad \forall  \, u \in \c1(\RN).
\end{align*}
\end{remark}

\noi \textbf{Proof of Theorem \ref{Lorentz product}:} \label{proof 1.7}  We consider two cases. 

\noi $\bm{\underline{q \in (0,p)}}$: Let $s,t \in (0,1)$. Let  $k>p$, $g_1 \in L^{\frac{k}{(k-p)st+p}, \frac{1}{st}}(\Om_1)$ and $g_2 \in L^{\frac{1}{st}}(\Om_2)$. By part $(i)$ of Theorem \ref{Symmetrization}, we have $g_1 \in  \mathcal{H}_{p,q}(\Om_1)$ with $q=(1-st)p$. Hence Proposition \ref{property2} infers $|g_1 g_2| = g \in \pq$. Moreover, using \eqref{eqn:Lorentz} and \eqref{property2-1}, we obtain  
\begin{align*}
    \int_{\Om}  |g(x)| |u(y,z)|^q \,  \dy \dz  & \le C \norm{g_1}_{\frac{k}{(k-p)st+p}, \frac{1}{st}} \int_{\Om_2} |g_2(z)| \left( \disone |\Gr_{y}   u(y,z)|^p \, \dy \right)^{1-st} \, \dz\\ & \le C \norm{g_1}_{\frac{k}{(k-p)st+p}, \frac{1}{st}} \norm{g_2}_{\frac{1}{st}} \left( \int_{\Om} |\Gr u(y,z)|^p \, \dy \dz \right)^{1-st}, \quad \forall \, u \in \c1(\Om),
    \end{align*}
where $C=C(k,p,q)>0$.

\noi $\bm{\underline{q \in [p,p^*]}}$: Let $s \in [0,1]$ and write $q=(1-st)p+stp^*$. If $t=0$, then we have $q=p$, $N-k>p$, $g_1 \in L^{\infty}(\Om_1)$ and $g_2 \in L^{\frac{N-k}{p},\infty}(\Om_2)$. By interchanging the roles of $\Om_1$ and $\Om_2$,  we obtain from Proposition \ref{cylin Lorentz prop} that  
\begin{align*}
  \dis |g_1(y)| |g_2(z)| |u(x)|^p \, \dx & \le C \norm{g_1}_{\infty}\norm{g_2}_{\frac{N-k}{p}, \infty} \dis |\Gr u(x)|^p \, \dx, \quad \forall \, u \in \c1(\Om).
\end{align*}
For $t=1$, we have $k>p,$ $q=(1-s)p+sp^*$, $g_1 \in L^{\frac{k}{(1-s)p},\infty}(\Om_1)$ and $g_2 \in L^{\infty}(\Om_2).$ Thus, again by Proposition \ref{cylin Lorentz prop},  
\begin{align*}
  \dis |g_1(y)| |g_2(z)| |u(x)|^q \, \dx & \le C \norm{g_1}_{\frac{k}{(1-s)p},\infty}\norm{g_2}_{\infty} \left( \dis |\Gr u(x)|^p \, \dx \right)^\frac{q}{p}, \quad \forall \, u \in \c1(\Om).
\end{align*}
 Next we consider the case $t \in (0,1)$ and $q=(1-st)p+stp^*$. In this case, we have both $k>p$ and $N-k>p$, $g_1 \in L^{\frac{k}{(1-s)pt}, \infty}(\Om_1)$ and $g_2 \in L^{\frac{N-k}{(1-t)p}, \infty}(\Om_2)$. Then $|g_1|^\frac{1}{t}\in L^{\frac{k}{(1-s)p}, \infty}(\Om_1)$ and $|g_2|^\frac{1}{1-t} \in L^{\frac{N-k}{p}, \infty}(\Om_2)$. Hence by Proposition \ref{cylin Lorentz prop}, we obtain that $|g_1|^\frac{1}{t} \in \H_{p,q_1}(\Om)$ with $ q_1=(1-s)p+(1-(1-s))p^*=(1-s)p+sp^*$ and $|g_2|^\frac{1}{1-t}\in \H_{p,p}(\Om)$. 
 Therefore, $(i)$ of Proposition \ref{property1} assures that $g = |g_1 g_2| = \abs{\abs{g_1}^{\frac{1}{t}}}^t \abs{\abs{g_2}^{\frac{1}{1-t}}}^{1-t} \in \H_{p,q}(\Om)$ since
 $$q=tq_1+(1-t)q_2=t((1-s)p+sp^*)+(1-t)p=(1-st)p+stp^*.$$ 
Moreover, using the H\"{o}lder's inequality and Proposition \ref{cylin Lorentz prop}, 
 \begin{align*}
    \dis |g_1(y)||g_2(z)||u(x)|^q \,  \dx & \le \left( \dis  |g_1(y)|^{\frac{1}{t}}  |u(x)|^{q_1} \, \dx \right)^t \left( \dis |g_2(z)|^{\frac{1}{1-t}} |u(x)|^p \, \dx  \right)^{1-t} \\
    & \le C \norm{g_1}_{\frac{k}{(1-s)tp}, \infty} \norm{g_2}_{\frac{N-k}{(1-t)p}, \infty} \left( \dis |\Gr u(x)|^p \, \dx \right)^{\frac{q}{p}}, \quad \forall \, u \in \c1(\Om),
\end{align*}
where $C=C(N,k,p,q)>0$.
\qed

Now we prove Theorem \ref{weighted Lebesgue product}.

\noi \textbf{Proof of Theorem \ref{weighted Lebesgue product}:}
Let $\Om = \Om_{a,b,S} \times \RNk$ and $\g_1, \g_2$ be as in Theorem \ref{weighted Lebesgue product}. Let $u \in \c1(\Om)$. For a fixed $\om \in S$ and $\tau \in (a,b)$, let $\vph(\tau)= u_z(\tau\omega)$.

\noi $\bm{\underline{q \in (0,p]}}$:  We express
\begin{align*}
&|\vph(r)|  \\
& = \left\{ \begin{array}{cc}
   \displaystyle \left| \int_a^r \vph'(\tau)  \tau^{\frac{k-1}{p}} \tau^{\frac{1-k}{p}} \, \d \tau \right| \le \left(\int_a^r \tau^{\frac{1-k}{p-1}} \, \d \tau  \right)^{\frac{1}{\p}} \left(\int_a^r \tau^{k-1} |\vph'(\tau)|^p \, \d \tau \right)^{\frac{1}{p}}, \; & \text{if} \ \ k < p; \\
   \displaystyle  \left|- \int_r^b \vph'(\tau)  \tau^{\frac{k-1}{p}} \tau^{\frac{1-k}{p}} \, \d \tau \right| \le \displaystyle \left(\int_r^b \tau^{\frac{1-k}{p-1}} \, \d \tau\right)^{\frac{1}{\p}} \left( \int_r^b \tau^{k-1} |\vph'(\tau)|^p \, \d \tau\right)^{\frac{1}{p}}, \;  & \text{if} \ \ k>p.
\end{array}
\right.  
\end{align*}
As $\vph'(\tau)=\Gr_y u_z(\tau\om) \cdot \om$ for each $\om \in S$ and $z \in \RNk$, we get
\begin{align}\label{sec1}
|u_z(r\om)|^q \le \left\{\begin{array}{ll}
    \displaystyle \left( \frac{p-1}{p-k} \right)^{\frac{q}{\p}} r^{\frac{(p-k)q}{p}} \left( \int_a^b \tau^{k-1} | \Gr_y u_z(\tau\om) |^p \, \d \tau \right)^{\frac{q}{p}}, \quad &\text{if} \ \ k<p;  \\
    \displaystyle \left( \frac{p-1}{k-p} \right)^{\frac{q}{\p}} r^{\frac{(p-k)q}{p}} \left( \int_a^b \tau^{k-1} | \Gr_y u_z(\tau\om)|^p \, \d \tau \right)^{\frac{q}{p}}, \quad &\text{if} \ \ k>p,
 \end{array} 
 \right. 
\end{align}
$\forall  \, u \in \c1(\Om)$. Thus for $q \in (0,p]$, by taking $\ga=0, \de = \frac{q}{p}$, and $h(r)= r^{\frac{(p-k)q}{p}}$ one can verify that the assumptions of Proposition \ref{observ1} are satisfied. Therefore, $g \in \pq$ for $q \in (0,p]$.

\noi $\bm{\underline{q \in (p, P^*(1)]}}$:  We write $q = tp + (1-t)P^*(1)$ for some $t \in [0,1)$. 
First, we consider $t=0$ (i.e., $q= P^*(1)$), $\g_1 \in L^{\infty}((a,b))$ and $\g_2 \in L^1((0, \infty))$. Let $u_y(z) = u(y,z)$ be the $y$-section of $u$. For a fixed $\om \in \S^{N-k-1}$, set $\vph_1(\tau)= u_y(\tau\omega)$ where $\tau \in (0,\infty)$. Then we make the following estimate: 
\begin{align*}
|\vph_1(r)|^{P^*(1)} & = - P^*(1) \displaystyle \int_r^{\infty} |\vph_1(\tau)|^{P^*(1)-1} \vph'_1(\tau) \tau^{\frac{N-k-1}{p}}\tau^{\frac{1+k-N}{p}} \,  \d \tau  \no \\
& \le P^*(1) \left( \int_r^{\infty}  \tau^{\frac{1+k-N}{p-1}} |\vph_1(\tau)|^{\left(P^*(1) - 1\right)\p} \, \d \tau \right)^{\frac{1}{\p}}\left( \int_r^{\infty} \tau^{N-k-1} \abs{\vph'_1(\tau)}^p\,  \d \tau \right)^{\frac{1}{p}} \\
& \le P^*(1) r^{1+k-N} \left( \int_r^{\infty}  \tau^{N-k-1} |\vph_1(\tau)|^{p^*} \, \d \tau \right)^{\frac{1}{\p}}\left( \int_r^{\infty} \tau^{N-k-1} \abs{\vph'_1(\tau)}^p\,  \d \tau \right)^{\frac{1}{p}}.
\end{align*}
From the above inequality, for each $\om \in \S^{N-k-1}$ and $y \in \Om_{a,b,S}$, we get
\begin{align*}
    |u_y(r\om)|^{P^*(1)} \le P^*(1) {r}^{1+k-N} & \left( \int_0^{\infty} \tau^{N-k-1} |u_y(\tau\om)|^{p^*} \, \d \tau  \right)^{\frac{1}{\p}} \\
    & \left( \int_0^{\infty} \tau^{N-k-1} \abs{\Gr_z u_y(\tau\om)}^p\, \d \tau \right)^{\frac{1}{p}}, \quad \forall \, u \in \c1(\Om). 
\end{align*}
For $q=P^*(1), \g_1 \in L^{\infty}((a,b))$ and $\g_2 \in L^1((0, \infty))$, by taking $\ga= \frac{1}{\p}, \de = \frac{1}{p}$, and $h(r)= r^{1+k-N}$, the assumptions of Proposition \ref{observ2} are satisfied. Therefore, $g = |g_1 g_2| \in \mathcal{H}_{p,P^*(1)}(\Om)$ and
\begin{align}\label{lebesgue cylin-1}
    \dis |g_1(y)| |g_2(z)| |u(x)|^{P^*(1)} \, \dx \le C \norm{\g_1}_{L^{\infty}((a,b))} & \norm{\g_2}_{L^1((0, \infty))} \left( \dis |\Gr u(x)|^p \, \dx \right)^{\frac{P^*(1)}{p}},
\end{align}
$\forall \, u \in \c1(\Om)$ and for some $C=C(N,k,p,q)>0$. For $q=p$, we have 
\begin{eqnarray} \label{lebesgue cylin-2}
    \dis |g_1(y)| |g_2(z)| |u(x)|^p \, \dx \le C \norm{\g_1}_{L^1((a,b), r^{p-1})} \norm{\g_2}_{L^{\infty}((0, \infty))} \dis |\Gr u(x)|^p \, \dx, 
\end{eqnarray}
$\forall \, u \in \c1(\Om)$ and for some $C=C(N,k,p)>0$. Now we consider $t \in (0,1)$, $\g_1 \in L^{\be(p,q)}((a,b), r^{p-1})$ and $\g_2 \in L^{\frac{\be(p,q)}{\be(p,q)-1}}((0,\infty))$. Since
\begin{align*}
    t = \frac{q-P^*(1)}{p- P^*(1)} = \frac{{p^*}-{N'}q}{{p^*}-{N'}p} = \frac{1}{\be(p,q)},
\end{align*} 
we have $\g_1^{\frac{1}{t}} \in L^1((a,b), r^{p-1})$ and $\g_2^{\frac{1}{1-t}} \in L^1((0, \infty))$. Hence from \eqref{lebesgue cylin-2} and \eqref{lebesgue cylin-1}, we get $|g_1|^{\frac{1}{t}} \in \pp$ and $|g_2|^{\frac{1}{1-t}} \in \mathcal{H}_{p,P^*(1)}(\Om)$ respectively. 
Therefore, using $(i)$ of Proposition \ref{property1}, we conclude that $g=|g_1 g_2| \in \pq$ with $q=tp+(1-t)P^*(1)$. Moreover, there exists $C=C(N,k,p,q)>0$ such that 
\begin{align*}
    \dis |g(x)| |u(x)|^q \, \dx  & \le \left( \dis |g_1(y)|^{\frac{1}{t}} |u(x)|^p \, \dx \right)^{t} \left( \dis |g_2(z)|^{\frac{1}{1-t}} |u(x)|^{P^*(1)} \, \dx \right)^{1-t} \\
    & \le C \norm{\g_1^{\frac{1}{t}}}_{L^1((a,b), r^{p-1})}^t \norm{\g_2^{\frac{1}{1-t}}}_{L^1((0, \infty))}^{1-t} \left( \dis |\Gr u(x)|^p \, \dx \right)^{\frac{q}{p}} \\
    & = C \norm{\g_1}_{L^{\be(p,q)}((a,b), r^{p-1})}  \norm{\g_2}_{L^{\frac{\be(p,q)}{\be(p,q)-1}}((0,\infty))} \left( \dis |\Gr u(x)|^p \, \dx \right)^{\frac{q}{p}}, \, \forall \, u \in \c1(\Om).
\end{align*}

\noi $\bm{\underline{q \in ( P^*(1), p^*]}}$: 
Let $q= t P^*(1) + (1-t)p^*$ for some  $t \in [0,1)$. Then $t = \frac{N}{\al(p,q)}$. For $t \in [0,1)$, $\g_1 \in L^{\infty}((0,\infty))$ and $\g_2 \in  L^{\frac{\al(p,q)}{N}}((a,b))$, using the H\"{o}lder's inequality, \eqref{lebesgue cylin-1} and the embedding $\Dp \hookrightarrow L^{p^*}(\Om)$, we similarly get $g= |g_1g_2| \in \pq$ for $q \in (P^*(1), p^*]$. We also get  
\begin{align*}
    \dis  |g(x)| |u(x)|^q \, \dx  \le C \norm{\g_1}_{L^{\infty}((0,\infty))} \norm{\g_2}_{L^{\frac{\al(p,q)}{N}}((a,b))}  \left( \dis \abs{\Gr u(x)}^p \, \dx \right)^{\frac{q}{p}}, \quad \forall \, u \in \c1(\Om),
\end{align*}
for some $C=C(N,k,p,q)>0$.
\qed

%% file: Compact.tex
\section{Compactness and the existence of solution}
 In this section, we prove the compactness of $G_q$ for $g$ as given in Theorem \ref{cpct1} and then show the existence of solutions to the problem \eqref{eqn:evp}. First, we prove the following compactness result:

\begin{proposition}\label{locally embedd}
Let $\Om$ be as given in \eqref{domain}. Then $\Dp $ is compactly embedded into $ L^q_{loc}(\Om)$ for $q \in (0, \de)$, where $\de=p^*$ (if $N>p$), and $\de=\infty$ (if $N \le p).$ 
\end{proposition}

\begin{proof}
For $u \in \Dp$, $|\Gr u| \in L^p(\Om)$. If $N>p$, then using the Sobolev embedding, we also obtain $u\in L_{loc}^p(\Om)$. In particular, for each $K$ compact set in $\Om,$ there exists $C=C(N,p,K) >0$ such that
\begin{align}\label{compact1}
    \int_{K} \left( |u(x)|^p + |\Gr u(x)|^p \right) \, \dx \le C \dis |\Gr u(x)|^p \, \dx, \quad \forall \, u \in \Dp.
\end{align}
Thus $\Dp \hookrightarrow W^{1,p}_{loc}(\Om)$. If $N \le p$, then using Corollary \ref{function space}, $\Dp \hookrightarrow W^{1,p}_{loc}(\Om)$. Further, by Rellich-Kondrachov compactness theorem, 
\begin{equation}\label{eq:loc}
W^{1,p}_{loc}(\Om) \hookrightarrow L^q_{loc}(\Om) \mbox{ compactly for  } q \in \left\{ \begin{array}{ll}
         [1,p^*], & \quad  \text{for} \ \ N>p; \\
         \intervals{[1,\infty)} & \quad  \text{for} \ \ N \le p.
\end{array}\right. 
\end{equation}
For $q\le 1,$   $u\in L^1_{loc}(\Om),$   using the H\"{o}lder's inequality with the conjugate pair $(\frac{1}{q}, \frac{1}{1-q})$, we get
\begin{align*}
     \int_K |u(x)|^q \, \dx  \le \left( \int_K |u(x)| \, \dx \right)^q |K|^{1-q},
 \end{align*}
 for every compact set $K$ in $\Om.$ Therefore, $L^1_{loc}(\Om) \hookrightarrow L^q_{loc}(\Om)$ and hence by \eqref{eq:loc}, we conclude that $W^{1,p}_{loc}(\Om) \hookrightarrow L^q_{loc}(\Om) \text{  compactly for  } q \in (0,1)$ as well. This completes the proof. 
\end{proof}

Indeed, the above proposition shows that the map $G_q$ is compact on $\Dp$ for $g=\chi_K$, where $\chi_K$ is the characteristic function of $K$ in $\Om.$ However, in the following lemma, we prove the compactness of $G_q$ for $g$ in a more general class of weight functions.

\begin{lemma} \label{cpctlemma}
Let $p \in (1, \infty)$. For $i=1,2,$ let $g_i \in \F_{X_i}:=\overline{\cc(\Om_i)}^{X_i}$ and \eqref{normineq} holds. Then the map $$G_q(u) = \int_{\Om} |g||u|^q, \quad \forall \, u \in \Dp,$$ is compact on $\Dp$ for $q \in (0, \de)$, where $\de=p^*$ (if $N>p$), and $\de=\infty$ (if $N \le p).$ 
\end{lemma}

\begin{proof}
Let $u_n \rightharpoonup u$ in $\Dp$ and let $\ep >0$ be given. Set 
$M= \sup\{ \norm{|\Gr u_n|}_{p}^q + \norm{|\Gr u|}_{p}^q \}$. For $g_i \in \overline{\cc(\Om_i)}^{X_i}$ ($i=1,2$), we split $g_i = g_{\ep,i} + (g_i - g_{\ep,i})$ where $g_{\ep,i} \in \cc(\Om_i)$ such that $\norm{g_i - g_{\ep,i}}_{X_i} < \frac{\ep}{M}$. Then we write
\begin{align}\label{compact-1}
 \no   \dis |g_1| |g_2| \abs{(\abs{u_n}^q - \abs{u}^q)}  \le \dis |g_1 - g_{\ep,1}| |g_2| & \abs{(\abs{u_n}^q - \abs{u}^q)} +  \dis |g_{\ep,1}| |g_2-g_{\ep,2}| \abs{(\abs{u_n}^q - \abs{u}^q)} \\
    & +  \dis |g_{\ep,1}| |g_{\ep,2}| \abs{(\abs{u_n}^q - \abs{u}^q)}.
\end{align}
We estimate the first two integrals in the right hand side of \eqref{compact-1}, using \eqref{normineq} as 
\begin{align}\label{compact-2}
 \no   \dis ( & |g_1 - g_{\ep,1}|  |g_2| + |g_{\ep,1}| |g_2-g_{\ep,2}| ) \abs{(\abs{u_n}^q - \abs{u}^q)} \\ 
 & \le C \left(  \norm{g_1 - g_{\ep,1}}_{X_1} \norm{g_2}_{X_2} + \norm{g_{\ep,1}}_{X_1} \norm{g_2-g_{\ep,2}}_{X_2} \right) \left( \norm{|\Gr u_n|}_{p}^q + \norm{|\Gr u|}_{p}^q \right).
\end{align}
Further, using Proposition \ref{locally embedd}, there exists $n_1 \in \N$ such that 
$$ \dis|g_{\ep,1}| |g_{\ep,2}| \abs{(\abs{u_n}^q - \abs{u}^q)} = \int_{K_1 \times K_2} |g_{\ep,1}| |g_{\ep,2}| \abs{(\abs{u_n}^q - \abs{u}^q)} < \ep, \quad  \forall  \, n \ge n_1,$$
where $K_i \subset \Om_i$ is the compact support of $g_{\ep,i}$. Therefore, from \eqref{compact-1} and \eqref{compact-2}, $$ \dis |g|\abs{(\abs{u_n}^q - \abs{u}^q)} =  \dis |g_1| |g_2| \abs{(\abs{u_n}^q - \abs{u}^q)} < C\ep, \quad  \forall  \, n \ge n_1. $$ Thus, $G_q(u_n) \rightarrow G_q(u)$ as $n \rightarrow \infty$. 
\end{proof}

\noi {\bf{Proof of Theorem \ref{cpct1}:}} The compactness of $G_q$ follows from Lemma \ref{cpctlemma}. Now for $q>1$, we show the existence of non-negative solution to the problem \eqref{eqn:evp}.  Recall that 
\begin{align}\label{exists-1}
    \frac{1}{B_q(g)} = \inf\left\{ J(u)  : u \in N_g \right\} = \inf\left\{ R(u)  : u \in \Dp \setminus \{0\} \right\},
\end{align}
where $J(u) = \int_{\Om} |\Gr u|^p, N_g = \{ u \in \Dp :  G_q(u) = 1 \},$ and $R(u) = (\int_{\Om} g |u|^q )^{-\frac{p}{q}} \int_{\Om} |\Gr u|^p$. Let $(u_n)$ be a minimizing sequence for $J$ on the set $N_g$. By the coercivity of $J$, the sequence $(u_n)$ is bounded in $\Dp$ and hence admits a subsequence $(u_{n_k})$ such that $u_{n_k} \rightharpoonup u_1$ in $\Dp.$ Now using the compactness of $G_q$, we have $u_1 \in N_g.$ Further, the weak lowersemicontinuity of the norm $\norm{\Gr (\cdot)}_p$ gives
\begin{align*}
    \frac{1}{B_q(g)} = \lowlim_{k \rightarrow \infty} \dis \abs{\Gr u_{n_k}}^p \ge \dis \abs{\Gr u_1}^p \ge \frac{1}{B_q(g)}.
\end{align*}
Therefore, $\frac{1}{B_q(g)}$ is attained and $J$ admits a minimizer $u_1$ over $N_g$. Moreover, from \eqref{exists-1}, $u_1$ also minimizes $R$ over $\Dp \setminus \{0\}$, and  hence $\left<R'(u_1), v  \right>=0$ for $v \in \Dp$. Therefore, as $u_1 \in N_g$ and $q>1$,  we obtain 
\begin{align}\label{weak formulation}
     \dis |\Gr u_1|^{p-2} \Gr u_1 \cdot \Gr v = \frac{1}{B_q(g)} \dis  g \abs{u_1}^{q-2}u_1 v, \quad \forall  \, v \in \Dp.
\end{align} 
Since $|u_1| \in \Dp$ and $R(|u_1|)=R(u_1)$, we easily see that   $|u_1|$ is a non-negative solution of \eqref{eqn:evp}. 
\qed  

\begin{remark}
Let $N>p$. For $v \in \cc(\Om)$ with $v \ge 0$, we have
\begin{align*}
    \dis |\Gr (|u_1|)|^{p-2} \Gr (|u_1|) \cdot \Gr v = \frac{1}{B_q(g)} \dis  g \abs{u_1}^{q-1} v \ge 0.
\end{align*}
Then for $q \in [p,p^*)$, $|u_1| \in \Dp$ satisfies all the assumptions of Proposition \cite[Proposition 3.2]{KLP}. Therefore, by the strong maximum principle, $|u_1| >0$ a.e. in $\Om$. 
\end{remark}

The following proposition shows that $G_q$ is not compact for the cylindrical weight $g(x) = |y|^{-\frac{N}{\al(p,q)}}.$ 

\begin{proposition}\label{noncmpct}
Let $q \in [p,p^*]$ and $0 \in \overline{\Om}$. Then the map $G_q$ is not compact for $g(x)=|y|^{-\frac{N}{\al(p,q)}}$.
\end{proposition}

\begin{proof}
Let $q \in [p,p^*]$, $s=\frac{N}{\al(p,q)}$ and $0 \in \overline{\Om}$. Then using the Maz'ja's criteria \cite[Section 2.4.2, Theorem 1, page 130]{Mazja}, one can verify that the map $G_q$ is not compact for $|x|^{-s}$. Now suppose,  $G_q$ is compact for $|y|^{-s}$. Then for a sequence $u_n \wra u$ in $\Dp$, $\int_{\Om} \frac{|u_n|^q}{|y|^s} \ra \int_{\Om} \frac{|u|^q}{|y|^s} $, as $n \ra \infty.$ By Br\'ezis-Lieb lemma \cite{BL},
$$ \lim_{n \ra \infty} \left|\int_{\Om} \frac{|u_n|^q}{|y|^s} - \int_{\Om} \frac{|u|^q}{|y|^s} - \int_{\Om}  \frac{|u_n -u|^q}{|y|^s}  \right| = 0.$$ 
Therefore, $\int_{\Om} \frac{|u_n -u|^q}{|y|^s} \ra 0 $, as $n \ra \infty.$ Further, using $|x|^{-s} \leq |y|^{-s}$, we get $\int_{\Om} \frac{|u_n -u|^q}{|x|^s} \ra 0 $ as $n \ra \infty,$ a contradiction. Therefore, $G_q$ is not compact for $|y|^{-s}$.
\end{proof}

\begin{remark}
For $q \in [p,p^*]$ and $g(x) = |y|^{- \frac{N}{\al(p,q)}}$, using the concentration compactness principles, the existence of the solutions of \eqref{eqn:evp} is obtained on $\Om =\R^N$ \cite[Remark 2.6]{BT}.
\end{remark}

%% file: ExampleandRemarks.tex
\section{Examples and concluding remarks} In this section, we provide examples to show that the functions spaces given by Theorem \ref{Symmetrization}, Theorem \ref{weighted Lebesgue} and Theorem \ref{weighted Lebesgue product} are mutually independent. We also prove the necessary conditions.  

\begin{example} $(i)$ The spaces $L^{\al(p,q),r}(\R^N)$ and $L^1((0,\infty),r^{\frac{N}{\al(p,q)}-1})$ are not comparable. For $N>p$ and $q \in [p,p^*)$, consider the following functions on $\RN$: \begin{align*}
    g_1(x) = |x|^{-\frac{N}{\al(p,q)}} \quad \text{and} \quad g_2(x) = (|x|+1)^{-\left( \frac{N}{\al(p,q)}+1 \right)}.
\end{align*}
Then $g_1 \in L^{\al(p,q), \infty}(\RN)$ and $g_2 \notin L^{\al(p,q), \infty}(\RN)$.  Further, $\g_1 \notin L^1((0, \infty), r^{\frac{N}{\al(p,q)}-1})$ and  $\g_2 \in L^1((0, \infty), r^{\frac{N}{\al(p,q)}-1})$. 

\noi $(ii)$ The function spaces provided by Theorem \ref{weighted Lebesgue} and Theorem \ref{weighted Lebesgue product} are independent. For instance, Theorem \ref{weighted Lebesgue product}  provides weight functions on the domain $\Om=\left(\R^2\setminus B_1(0)\right)\times \R$, for which Theorem \ref{weighted Lebesgue} is not applicable. On the other hand, consider the following function:
\begin{align*}
    g(x)= \left\{ \begin{array}{cc}
        |x|^{-\frac{1}{2}}, & |x| \leq 1;  \\
        0, & \mbox{otherwise}.
    \end{array} \right.
\end{align*}
Since $\g \in L^1((0, \infty))$, we can apply Theorem \ref{weighted Lebesgue} to show $g \in \mathcal{H}_{p, P^*(1)}(\RN)$. Although, Theorem \ref{weighted Lebesgue product} (for $q=P^*(1)$) is not applicable as $\g \notin L^{\infty}((0, \infty)).$ 
\end{example}

Next we see that not all products of $g_1, g_2$ give rise to $(p,q)$-Hardy potentials.
\begin{example}\label{non-compatible}
For $q \in (0, P^*(k))$, we consider $g_1(y)=|y|^{-\frac{N}{\al(p,q)}}, y \in \Rk$. Hence $g_1^*(t) = \left( \frac{\om_k}{t} \right)^{\frac{N}{\al(p,q)k}}$ and $|g_1|_{\frac{\al(p,q)k}{N},\infty} = \om_k^{\frac{N}{\al(p,q)k}}$. Therefore, $g_1 \in L^{\frac{\al(p,q)k}{N},\infty}(\R^k)$. Since $g_1$ is not locally integrable on $\Rk$,  $g(x)=g_1(y) g_2(z) \notin \mathcal{H}_{p,q}(\R^N)$ for $q \in (0, P^*(k))$ and for any non-zero $g_2\in L^1_{loc}(\R^{N-k}).$ \end{example}

In the following remark, we describe the connection between Fefferman-Phong type conditions and Theorem \ref{Symmetrization}.

\begin{remark} \label{connection}
\noi $(i)$ For $N > p$ and $q\in [p,p^*]$, every weight function in $L^{\al(p,q), \infty}(\RN)$ satisfies \eqref{SW1} (by Remark \ref{equivalent}) and hence $L^{\al(p,q), \infty}(\RN) \subset \pqR$. 
This gives an alternate proof for Theorem \ref{Symmetrization} -$(i)$ without involving the embedding of $\mathcal{D}^{1,p}_0(\R^N)$.

\noi $(ii)$ We claim that for $N>p$ and $q>p^*,$ if $g \in L^s_{loc}(\R^N)$ satisfies \eqref{SW1}, then $g \equiv 0$. We choose $x_0 \in \R^N$ and consider a sequence $(Q_n)$ of cubes centred at $x_0$ and $|Q_n| \ra 0$ as $n \ra \infty$. Since $g$ satisfies \eqref{SW1},
 \begin{align} \label{SW2}
     \lim_{n \ra \infty} |Q_n|^{\frac{s}{\al(p,q)}} \left(\frac{1}{|Q_n|}\int_{Q_n} |g(x)|^s \, \dx  \right) \leq c_1,
 \end{align}
where $c_1=C_1(p,q)>0$. By the Lebesgue-Besicovitch differentiation theorem, 
 $$\lim_{n \ra \infty} \frac{1}{|Q_n|} \int_{Q_n} |g(x)|^s \, \dx  = |g(x_0)|^s.$$ Moreover, since $\al(p,q)<0$, we have $|Q_n|^{\frac{s}{\al(p,q)}} \ra \infty$, as $n \ra \infty$. Thus, \eqref{SW2} ensures that $g(x_0)=0.$ Since $x_0 \in \R^N$ is arbitrary, we get $g = 0$ a.e. in $\RN$.

\noi $(iii)$ Let $N \le p \leq q$, $\Om$ be bounded and $g \in L_{loc}^1(\Om)$. Let $g_{ext}$ be the zero extension of $g$ outside $\Om$, such that $g_{ext}$ satisfies \eqref{SW1}. Choose $Q_0 \supseteq \Om$. Then
 \begin{align*}
  |Q_0|^{\frac{s}{\al(p,q)} -1} \int_{Q_0} |g_{ext}(x)|^s \, \dx \leq    \sup_{ \{ Q \subset \RN, |Q| < \infty \}} |Q|^{\frac{s}{\al(p,q)} -1} \int_{Q} |g_{ext}(x)|^s \, \dx   \le c_1. 
 \end{align*}
 Consequently, $g \in L^s(\Om)$. We would like to point out that $L^s(\Om) \subseteq L^{1,\infty,\frac{q}{p'}}(\Om) \subseteq L^1(\Om)$ (by $(ii)$ and $(iii)$ of Proposition \ref{LZ prop}). 
Thus, for a bounded domain, Theorem \ref{Symmetrization} -$(ii),(iii)$ gives a bigger class of weight functions than the Sawyer's condition \eqref{SW1}.
\end{remark}

\begin{remark}
The weight functions provided in this article do not exhaust  the entire $\pq.$  For example, consider the weight functions of the form $$g(x)=g_1(y)g_2(z)g_3(w),\;x=(y,z,w) \in \R^k\times \R^l \times \R^{N-k-l},$$ where $1\le k,l \le \N; k+l \le N.$ In general, corresponding to  $s_i\in \{1,2,...N\}$ with $1\le i\le m\le N$ and $\sum_{i=1}^m s_i=N,$  consider the weight functions of the form $$g(x)=\prod_{i=1}^mg_i(y_i),\; x=(y_1,y_2\ldots y_m)\in \prod_{i=1}^m \R^{s_i}.$$ One can provide conditions on $g_1,g_2\ldots,g_m$ so that $g$ is in $\pq.$
\end{remark}

\subsection{The necessary conditions}
In the following proposition we show under certain conditions on $g$ that the spaces mentioned in Proposition \ref{HSLo} and Proposition \ref{HSLo-1} are necessary for $g$ to satisfy \eqref{p-qHardy}. A similar result for the Hardy-Rellich inequalities is obtained in \cite{AUA}. 

\begin{proposition} \label{necessarycond}
Let $\Om =B_R(0)$ with
\begin{align*}
R \in \left\{\begin{array}{ll}
    (0,\infty], \quad &\text{if} \ \ N>p; \\
    (0,\infty), \quad &\text{if} \ \ N=p.
 \end{array} \right.
\quad  \mbox{and} \quad
q \in \left\{\begin{array}{ll}
    [p,p^*], \quad &\text{if} \ \ N>p; \\
    (p,\infty), \quad &\text{if} \ \ N=p.
 \end{array} \right.
\end{align*}
If $g \in \pq$ is radial, radially decreasing, then
\begin{align*}
g \in X:= \left\{\begin{array}{ll}
    L^{\al(p,q),\infty}(\Om), \quad &\text{if} \ \ N>p; \\
    L^{1,\infty,\frac{q}{N'}}(\Om), \quad &\text{if} \ \ N=p.
 \end{array} \right.\end{align*}
\end{proposition}
  
\begin{proof}
\underline{$\bm{N>p}$}: Let $ g \in \pq$ be a radial and radially decreasing. For each $r\in (0,R),$ consider the following function:
 \begin{align*}
           u_r(x)= \left\{\begin{array}{ll}
    r-|x|, \quad &\text{for} \ \ |x|\le r; \\
    0, \quad &\text{for} \ \ |x| \ge r.
 \end{array} \right. \end{align*}
Clearly,
  \begin{align}
          \Gr u_r(x)= \left\{\begin{array}{ll}
    -\frac{x}{|x|}, \quad &\text{for} \ \ |x|\le r; \\
    0, \quad &\text{for} \ \ |x| \ge r,
 \end{array} \right. \quad \text{and} \quad \dis |\Gr u_r(x)|^p \, \dx = \int_{B_r} \dx = \om_N r^N. \label{necessary2} \end{align}
Thus for each $r\in (0,R),$ $u_r\in \Dp.$ Furthermore, since $g \in \pq$, 
 \begin{equation}\label{necessary3}
  \int_\Om g(x)|u_r(x)|^q \, \dx \leq C \left(\int_\Om |\nabla u_r(x)|^p \, \dx \right)^{\frac{q}{p}},\; \forall r\in(0,R).
 \end{equation}
 Since $g$ is  radial and radially decreasing, we estimate the left hand side of the above inequality as below:        
  \begin{align*}
   	\int_{\Om} g(x) |u_r(x)|^q \, \dx \geq \int_{B_{{\frac{r}{2}}}}g(|x|) |u_r(x)|^q \, \dx & \geq  \left(r-{\frac{r}{2}}\right)^q\int_{B_{{\frac{r}{2}}}}g(|x|)\, \dx  \\
   	   & =\left({\frac{r}{2}}\right)^q \int_0^{\om_N ({\frac{r}{2}})^N} g^* (s)\, \ds.  
  	\end{align*}  
Therefore,  from \eqref{necessary2} and \eqref{necessary3} we obtain
    $$\left({\frac{r}{2}}\right)^q \int_0^{\om_N ({\frac{r}{2}})^N} g^*(s) \, \ds \le C  r^{\frac{Nq}{p}}.$$ 
Now by setting $\om_N ({\frac{r}{2}})^N=t $ and since $0<r<R$ is arbitrary, we conclude that $$ \sup_{t \in (0,\frac{|\Om|}{2^N})} t^{{\frac{1}{\al(p,q)}}} g^{**}(t) \leq C.$$ 
Moreover, $ t^{{\frac{1}{\al(p,q)}}} g^{**}(t) $ is bounded on $ (\frac{|\Om|}{2^N}, |\Om|) $. Therefore,   $g$ must belong to $L^{\al(p,q),\infty}(\Om)$. 
  
\noi \underline{$\bm{N=p}$}:  Let $R<\infty$ and $q \in [N,\infty)$. For   each $r\in (0,R),$ we consider the following function:
\begin{align*}
 u_r(x)= \left\{\begin{array}{ll}
    \log \left(\frac{R}{r} \right), \quad &\text{for} \ \ |x|\le r; \\
    \log \left(\frac{R}{|x|} \right), \quad &\text{for} \ \ |x| \ge r.
 \end{array} \right.
\end{align*}
Clearly,
\begin{align}
           \Gr u_r(x)= \left\{\begin{array}{ll}
    0, \quad &\text{for} \ \ |x|\le r; \\
    -\frac{x}{|x|^2}, \quad &\text{for} \ \ |x| \ge r,
 \end{array} \right. \quad \text{and} \quad 
 \int_{\Om} |\nabla u_r(x)|^N \, \dx &=  \om_N \int_{r}^R \, \frac{\dt}{t} =  \om_N  \log \left(\frac{R}{r} \right).  \label{necessary4} \end{align}
Thus for each $r\in (0,R),$ $u_r\in \DN.$ Furthermore, since $g \in \pq$, 
 \begin{equation}\label{necessary5}
  \int_\Om g(x)|u_r(x)|^q \, \dx \leq C \left(\int_\Om |\nabla u_r(x)|^N \, \dx\right)^{\frac{q}{N}}, \; \forall r \in (0,R).
 \end{equation}
 Since $g$ is  radial and radially decreasing, we estimate the left hand side of the above inequality as below:        
   \begin{align*}
   	\dis g(x) |u_r(x)|^q \dx  \ge  \left(\log \left(\frac{R}{r}\right)\right)^q\int_{B_{r}}g(|x|) \, \dx \notag 
   	 = \left(\log \left( \frac{R}{r} \right) \right)^q \int_0^{\om_N r^N} g^*(s) \, \ds.   
  	\end{align*}   
Now  using \eqref{necessary4} and \eqref{necessary5} we obtain
$$ \int_0^{\om_N r^N} g^*(s)\ds \leq C \left(\log \left( \frac{R}{r} \right) \right)^{-\frac{q}{N'}} .$$ 
By setting $\om_N r^N=t $ and since $0<r<R$ is arbitrary, we conclude that $$ \sup_{t \in (0,|\Om|)} t \left( \log \left(\frac{|\Om|}{t} \right) \right)^{\frac{q}{N'}} g^{**}(t) \leq C.$$ Therefore, $g \in L^{1, \infty, \frac{q}{N'}}(\Om)$. 
\end{proof}  

%% file: Appendix.tex
\section{}
In this section, we provide various Lorentz and Lorentz-Zygmund spaces in $\pq$. Then we supply alternative proofs for the Lorentz-Sobolev and Brezis-Wainger embeddings. First, we state a sufficient condition for the one-dimensional weighted Hardy inequalities due to Muckenhoupt in \cite[Theorem 2]{Muckenhoupt} (for $q=p$), also see, \cite[Theorem 2]{Bradley} (for $q \ge p$),  \cite{Mazja, Sinnamon} (for $0 < q <p$). For further readings on these inequalities, we refer to \cite[Chapter 5]{KMP}.
\begin{proposition}[Muckenhoupt condition]\label{Mucken}
For $b \in (0, \infty],$ let $v,w$ be non-negative measurable functions on $(0,b)$ with $w > 0$. Let $p\in (1,\infty),q\in (0,\infty)$, and $\ga = \frac{pq}{p-q}.$
\begin{enumerate}[(i)]
   \item If $0<q<1$ and 
   $$ \displaystyle A_1 = \left( \int_0^b \left( \int_0^s v(t) \, \dt \right)^{\frac{\ga}{p}} \left( \int_s^b w(t)^{1-\p} \, \dt \right)^{\frac{\ga}{\p}} v(s) \, \ds \right)^{\frac{1}{\ga}} < \infty,$$
  then
 \begin{align}\label{Mucken1}
\displaystyle \left( \int_0^b \left | \int_s^b f(t) \, \dt \right|^q v(s) \, \ds \right)^{\frac{1}{q}} \leq {(\p)}^{\frac{1}{\ga}} q^{\frac{1}{p}}   A_1 \left( \int_0^b | f(s) |^p w(s) \, \ds \right)^{\frac{1}{p}}
\end{align}
holds for any measurable function $f$ on $(0,b).$
   \item If $1 \le q <p < \infty$ and 
 $$ \displaystyle A_2 = \left( \int_0^b \left( \int_0^s v(t) \, \dt  \right)^{\frac{\ga}{q}} \left( \int_s^b w(t)^{1- \p} \, \dt \right)^{\frac{\ga}{\q}} w(s)^{1-\p} \, \ds  \right)^{\frac{1}{\ga}} < \infty,  $$
then
 \begin{align}\label{Mucken2}
\displaystyle \left( \int_0^b \left | \int_s^b f(t) \, \dt \right|^q v(s) \, \ds \right)^{\frac{1}{q}} \leq  (\p)^{\frac{1}{\q}}  q^{\frac{1}{q}} A_2 \left( \int_0^b | f(s) |^p w(s) \, \ds \right)^{\frac{1}{p}}
\end{align}
holds for any measurable function $f$ on $(0,b).$
    \item If $1 \leq p \leq q < \infty$ and 
    $$ \displaystyle A_3 = \sup_{0<t<b} \left( \int_{0}^t v(s) \,  \ds \right)^{\frac{1}{q}} \left( \int_t^b w(s)^{1 - \p} \, \ds \right)^{\frac{1}{\p}} < \infty, $$
then 
\begin{align}\label{Mucken3}
\displaystyle \left( \int_0^b \left | \int_s^b f(t) \, \dt \right|^q v(s) \, \ds \right)^{\frac{1}{q}}  \le (\p)^{\frac{1}{\p}} p^{\frac{1}{q}} A_3 \left( \int_0^b | f(s) |^p w(s) \, \ds \right)^{\frac{1}{p}}
\end{align}
holds for any measurable function $f$ on $(0,b).$
\end{enumerate}
\end{proposition} 

\begin{proposition}\label{HSLo}
Let $N > p$ and  
\begin{align*}
 X:= \left\{ \begin{array}{ll}
         L^{\al(p,q),\frac{p}{p-q}}(\Om), & \quad  \text{for} \ \ q\in (0,p);\\
            L^{\al(p,q), \infty}(\Om), & \quad  \text{for} \ \ q \in [p,p^*]. \\
             \end{array}\right. 
\end{align*}
If $g \in X$, then  there exists $C = C(N,p,q)>0$ such that 
\begin{equation}\label{HSLo2.1}
    \displaystyle \int_0^{|\Om|} g^*(t) {u^*(t)}^q \,  \dt  \leq  C \norm{g}_{X}  \left( \int_{0}^{|\Om|} t^{p - \frac{p}{N}} | {u^*}^{\prime}(t) |^p \,  \dt \right)^{\frac{q}{p}}, \quad  \forall  \, u \in \c1(\Om).
\end{equation}  
\end{proposition}

\begin{proof}
In Proposition \ref{Mucken} we set $f = {u^*}^{\prime}, v = g^*$ and $w(t)= t^{p - \frac{p}{N}}.$ Then we calculate 
$$\displaystyle \int_{0}^s v(t) \, \dt = s \left( g^{**}(s) \right) \; \text{and} \; \displaystyle \int_s^{|\Om|} {w(t)}^{1 - \p} \, \dt \le \frac{N(p-1)}{N-p} s^{\frac{p-N}{N(p-1)}}.$$  Let $C_1(N,p) = \frac{N(p-1)}{N-p}.$ We consider three cases.

\noi $\bm{\underline{q \in (0,1)}}$: In this case,
\begin{align*}
    A_1^{\ga} &= \int_0^{|\Om|} \left( \int_0^s v(t) \, \dt \right)^{\frac{\ga}{p}} \left( \int_s^{|\Om|} w(t)^{1-\p} \, \dt \right)^{\frac{\ga}{\p}} v(s) \, \ds \\
    & \le  C_1^{\frac{\ga}{\p}} \int_0^{|\Om|}   (s g^{**}(s))^{\frac{\ga}{p}} s^{\frac{\ga(p-N)}{Np}} g^*(s) \, \ds \le  C_1^{\frac{\ga}{\p}} \int_0^{|\Om|} s^{\frac{\ga}{N}} (g^{**}(s))^{1+\frac{\ga}{p}} \, \ds,
\end{align*}
where $\ga=\frac{qp}{p-q}$. Therefore, 
\begin{align*}
    A_1 \le C_1^{\frac{1}{\p}} \left( \int_0^{|\Om|} s^{\frac{qp}{N(p-q)}} (g^{**}(s))^{\frac{p}{p-q}} \, \ds  \right)^{\frac{p-q}{qp}} = C_1^{\frac{1}{\p}}\norm{g}_{\al(p,q), \frac{p}{p-q}}^{\frac{1}{q}}.
\end{align*}
Thus for $g\in L^{\al(p,q),\frac{p}{p-q}}(\Om)$, the Muckenhoupt condition ($(i)$ of Proposition \ref{Mucken}) is satisfied. Therefore, using  \eqref{Mucken1} we obtain 
$$ \displaystyle  \int_0^{|\Om|} g^*(t) {u^*(t)}^{q} \, \dt   \le C \|g \|_{\al(p,q),  \frac{p}{p-q}}  \left( \int_{0}^{|\Om|} t^{(p - \frac{p}{N})} | {u^*}^{\prime}(t)|^p \,  \dt\right)^{\frac{q}{p}},$$
for some $C=C(N,p,q)>0$.
 
\noi $\bm{\underline{q \in [1,p)}}$:  In this case, we calculate
\begin{align}
  A_2^{\ga}  & = \int_0^{|\Om|} \left( \int_0^s v(t) \, \dt  \right)^{\frac{\ga}{q}} \left( \int_s^{|\Om|} w(t)^{1- \p} \, \dt \right)^{\frac{\ga}{\q}} w(s)^{1-\p} \, \ds  \no \\
  & \le  C_1^{\frac{\ga}{\q}} \int_0^{|\Om|}  (s g^{**}(s))^{\frac{\ga}{q}} s^{\frac{1}{N(p-1)} \left( \frac{\ga(p-N)}{\q} - p(N-1)\right)} \, \ds, \label{m2}
\end{align}
where $\ga = \frac{qp}{p-q}$. Moreover,  $$\frac{\ga}{q} + \frac{1}{N(p-1)} \left( \frac{\ga(p-N)}{\q} - p(N-1)\right) =  \frac{qp}{N(p-q)}.$$ Therefore, using \eqref{m2}, $A_2 \le C_1^{\frac{1}{\q}} \norm{g}_{\al(p,q), \frac{p}{p-q}}^{\frac{1}{q}}$. Thus for $g\in L^{\al(p,q),\frac{p}{p-q}}(\Om)$, the Muckenhoupt condition ($(ii)$ of Proposition \ref{Mucken}) is satisfied, and hence using \eqref{Mucken2} we obtain \eqref{HSLo2.1}. 

\noi $\bm{\underline{q \in [p,p^*]}}$: In this case, 
\begin{align*}
A_3 & = \sup_{s\in (0,|\Om|)} \left( \int_{0}^s v(t) \,  \dt \right)^{\frac{1}{q}} \left( \int_{s}^{|\Om|} w(t)^{1 - \p} \,  \dt \right)^{\frac{1}{\p}} \\
& \le C_1 \sup_{s\in (0,|\Om|)} \left( s g^{**}(s) \right)^{\frac{1}{q}} s^{\frac{p-N}{Np}} =  C_1 \left( \sup_{s \in (0,|\Om|)} g^{**}(s) s^{\frac{N(p-q) + qp}{Np}} \right)^{\frac{1}{q}} = C_1 \|g \|_{ \al(p,q), \infty}^{\frac{1}{q}}.
\end{align*}
Now for $g \in L^{\al(p,q), \infty}(\Om)$ using \eqref{Mucken3} we obtain \eqref{HSLo2.1}. 
\end{proof}

\begin{proposition}\label{HSLo-1}
Let $N = p$ and $\Om$ be bounded. Let  
\begin{align*}
 X:= \left\{ \begin{array}{ll}
         L^{1,\frac{N}{N-q};\frac{q}{N'}}(\Om), & \quad  \text{for} \ \ q\in (0,1); \\
         L^{1, \frac{N}{N-q}; q-1}(\Om), & \quad  \text{for} \ \ q\in [1,N); \\
         L^{1, \infty;\frac{q}{N'}}(\Om), & \quad  \text{for} \ \ q \in [N,\infty). \\
             \end{array}\right. 
\end{align*}
If $g\in X$, then  there exists $C = C(N,q)>0$ such that 
\begin{align}\label{HSLo2.2}
    \displaystyle \int_0^{|\Om|} g^*(t) {u^*(t)}^q \,  \dt  \leq  C \norm{g}_{X}  \left( \int_{0}^{|\Om|} t^{N-1} | {u^*}^{\prime}(t) |^N \,  \dt \right)^{\frac{q}{N}}, \quad  \forall  \, u \in \c1(\Om).
\end{align}  
\end{proposition}

\begin{proof}
We only consider the case where $q \in [N, \infty)$. For the other cases proof follows using similar set of arguments. As before, we set $f = {u^*}^{\prime}, v = g^*$ and $w(t)= t^{N-1}.$ We see that $ \int_s^{|\Om|} {w(t)}^{1 - N'} \,  \dt \le \log (\frac{e|\Om|}{s})$, and compute
\begin{align*}
A_3  = \sup_{s\in (0,|\Om|)} \left( \int_{0}^s v(t) \,  \dt \right)^{\frac{1}{q}} & \left( \int_{s}^{|\Om|} w(t)^{1 - N'} \,  \dt \right)^{\frac{1}{N'}}  \le \sup_{s\in (0,|\Om|)} \left( s g^{**}(s) \right)^{\frac{1}{q}} \left( \log \left(\frac{e|\Om|}{s} \right)  \right)^{\frac{1}{N'}} \\
& =  \left( \sup_{s\in (0,|\Om|)} s g^{**}(s) \left( \log \left(\frac{e|\Om|}{s} \right)  \right)^{\frac{q}{N'}}  \right)^{\frac{1}{q}} = \|g \|_{ 1, \infty,\frac{q}{N'}}^{\frac{1}{q}}.
\end{align*}
Thus for $g\in L^{1, \infty;\frac{q}{N'}}(\Om)$,  the Muckenhoupt condition ($(iii)$ of Proposition \ref{Mucken}) is satisfied. Therefore, using \eqref{Mucken3} we obtain \eqref{HSLo2.2}.
\end{proof}

In the following theorem we provide simple alternate proofs for the Lorentz-Sobolev embedding ($N>p$) and the Brezis-Wainger embedding ($N=p$).

\begin{theorem}\label{embedding} 
Let $\Om$ be an open set in $\R^N$ and $p \in (1,N].$
\begin{enumerate}[(i)]
    \item \textbf{The Lorentz-Sobolev embedding}:  Let $N>p.$ Then $\Dp\hookrightarrow L^{p^*,p}(\Om),$ i.e., there exists $C = C(N , p) > 0$ such that 
  \begin{align}\label{Embedd1}
      \norm{u}_{p^*, p} \le C \norm{u}_{\Dp}, \quad \forall  \, u \in \Dp.
  \end{align}
  \item \textbf{The Brezis-Wainger embedding}: Let $N=p$ and $\Om$ be  bounded. Then $\Dp \hookrightarrow L^{\infty,p,-1}(\Om)$, i.e., 
  there exists $C = C(p) > 0$ such that 
  \begin{align}\label{Embedd1.1}
      \norm{u}_{\infty,p,-1} \le C \norm{u}_{\Dp}, \quad \forall  \, u \in \Dp.
  \end{align}
\end{enumerate}
\end{theorem}

\begin{proof} 
$(i)$ For $g = |x|^{-p} \in  L^{\frac{N}{p}, \infty}(\RN)$, from $(i)$ of Example \ref{ex rearrangement}, we have $g^*(t) = \left( \frac{\om_N}{t} \right)^{\frac{p}{N}}$ and $\norm{g}_{\frac{N}{p}, \infty}= \frac{N \om_N^{\frac{p}{N}}}{N-p}$. Then using \eqref{HSLo2.1},
 \begin{align*}
      \displaystyle \int_0^{\infty} t^{- \frac{p}{N}} {u^*(t)}^p  \, \dt  \le  C  \int_{0}^{\infty} t^{p - \frac{p}{N}} | {u^*}^{\prime}(t) |^p \, \, \dt, \quad \forall  \, u \in \mathcal{D}^{1,p}_0(\R^N),
 \end{align*}
where $C= C(N,p)>0$. Notice that the left hand side of the above inequality is $|u|_{p^*, p}$  and it is equivalent to $\norm{u}_{p^*, p}$. Thus from the  P\'{o}lya-Szeg\"{o} inequality (Proposition \ref{HL and PS}),  we obtain
\begin{align*}
    \norm{u}_{p^*, p} \le C  \norm{u}_{\dpR}, \quad \forall  \, u \in \dpR.
\end{align*}
For $u\in \Dp,$ where $\Om$ is an open set in $\R^N$, one can consider the zero extension of $u$ to $\R^N$ to get \eqref{Embedd1}.

\noi $(ii)$ Let $\Om=B_R(0)$ where $R>0$. We consider the following function
\begin{align*}
    g(x) = \left\{ \begin{array}{ll}
         |x|^{-p}\left( \log  \left( e \left(\frac{R}{|x|} \right)^p \right) \right)^{-p}, & \quad  \text{ for }  x \in B_{R_1}(0); \\
          (p|x|)^{-p}, & \quad  \text{ for }  x \in B_{R}(0) \setminus B_{R_1}(0), \\
              \end{array}\right.
\end{align*}
where $R_1 = R e^{\frac{1-p}{p}}$. Since $g$ is radial and radially decreasing on $B_R(0)$, using Example \ref{ex rearrangement}, 
\begin{align*}
    g^*(t) = \left\{ \begin{array}{ll}
         \frac{\om_p}{t}  \left( \log  \left( \frac{e|\Om|}{t} \right) \right)^{-p}, & \quad  \text{ for }  t \in (0, |B_{R_1}(0)|); \\
         p^{-p} \frac{\om_p}{t}, & \quad  \text{ for }  t \in (|B_{R_1}(0)|,  |B_{R}(0)|). \\
              \end{array}\right.
\end{align*}
From $(ii)$ of Example \ref{ex rearrangement}, $g \in L^{1,\infty;p}(B_R(0))$. Hence using $(ii)$ of Proposition \ref{LZ prop} and $(ii)$ of Theorem \ref{Symmetrization}, we conclude that $g \in \mathcal{H}_{p,p}(B_R(0))$.  Moreover, from \eqref{HSLo2.2} and the P\'{o}lya-Szeg\"{o} inequality, there exists $C=C(p)>0$ such that 
 \begin{align*}
      & \displaystyle \int_0^{|B_R(0)|}   \left(\frac{{u^*(t)}}{\log(\frac{e|B_R(0)|}{t})} \right)^p  \, \frac{\dt}{t} \\ & \le \int_0^{|B_{R_1}(0)|}  \left(\frac{{u^*(t)}}{\log(\frac{e|B_R(0)|}{t})} \right)^p  \, \frac{\dt}{t} + \int_{|B_{R_1}(0)|}^{|B_{R}(0)|} \frac{\left( u^*(t) \right)^p}{t} \, \dt \\
      & \le C \int_0^{|B_{R}(0)|} g^*(t) (u^*(t))^p \, \dt \\
      & \le C \int_{B_R(0)} |\Gr u(x)|^p \, \dx, \quad \forall \, u \in \D^{1,p}_0(B_R(0)),
\end{align*}
where the first inequality holds since $(\log(\frac{e|B_R(0)|}{t}))^{-1} \le 1$ for $t\le |B_R(0)|$. Notice that, the left hand side of the above inequality is $|u|_{\infty,p; -1}$ (equivalent to $\norm{u}_{\infty,p; -1}$). Therefore, 
\begin{align*}
    \norm{u}_{\infty,p; -1} \le C  \norm{u}_{\D^{1,p}_0(B_R(0))}, \quad \forall  \, u \in \D^{1,p}_0(B_R(0)).
\end{align*} 
Furthermore, every bounded open set $\Om$ is contained in $B_R(0)$ for some $R>0.$  Thus the extension by zero to $B_R(0)$ together with above inequality yields \eqref{Embedd1.1}.
\end{proof}
\section*{Acknowledgement}
The corresponding author acknowledges the Department of Science \& Technology, India, for the research grant DST/INSPIRE/04/2014/001865. We thank the anonymous reviewers for their valuable comments that have improved this article.